\documentclass[a4paper, 12pt]{amsart}

\usepackage[T1]{fontenc}
\usepackage[ngerman,english]{babel}
\usepackage{amsmath}
\usepackage{amsthm}
\usepackage{amssymb}
\usepackage{amsfonts}
\usepackage{tabularx}
\usepackage{multirow}

\usepackage{tikz}
\usetikzlibrary{matrix,arrows,patterns,shapes}
\usepackage{MnSymbol} 
\usepackage{nicefrac}
\usepackage{paralist} 

\theoremstyle{plain}
\newtheorem{theorem}{\bf{\textsc{Theorem}}}[section]

\newtheorem{theoremAlpha}{\bf{\textsc{Theorem}}}

\newtheorem{corollary}[theorem]{\bf{\textsc{Corollary}}}
\newtheorem{lemma}[theorem]{\bf{\textsc{Lemma}}}
\newtheorem{proposition}[theorem]{\bf{\textsc{Proposition}}}

\numberwithin{equation}{section}

\theoremstyle{definition}

\theoremstyle{definition}
\newtheorem{remark}[theorem]{\normalfont{\textsc{Remark}}}

\newcommand{\CC}{\mathbb{C}}				
\newcommand{\RR}{\mathbb{R}}				
\newcommand{\NN}{\mathbb{N}}				

\newcommand{\Poly}{\mathcal P}			

\DeclareMathOperator{\End}{End}			
\DeclareMathOperator{\Hom}{Hom}			
\DeclareMathOperator{\GL}{GL}				
\DeclareMathOperator{\U}{U}					
\DeclareMathOperator{\SU}{SU}				
\DeclareMathOperator{\Sp}{Sp}				
\DeclareMathOperator{\SO}{SO}				

\DeclareMathOperator{\Id}{Id}				
\DeclareMathOperator{\Ad}{Ad}				

\renewcommand{\Im}{\operatorname{Im}}			
\DeclareMathOperator{\rank}{rk}			
\DeclareMathOperator{\Tr}{Tr}				
\DeclareMathOperator{\Lie}{Lie}			
\DeclareMathOperator{\Det}{Det}			
\newcommand{\linebundle}{\mathcal L} 

\renewcommand{\b}[1]{{\bf #1}}
\newcommand{\Vp}{{\mf n^+}}
\newcommand{\Vm}{{\mf n^-}}

\newcommand{\at}[1]{\big\rvert_{#1}}								
\newcommand{\sP}[2]{\langle#1|#2\rangle}						
\newcommand{\JTP}[3]{\left\{#1,\,#2,\,#3\right\}}		
\newcommand{\B}[2]{B_{#1,\,#2}}											
\newcommand{\mf}[1]{{\mathfrak{#1}}} 								
\newcommand{\JP}[1]{\text{\sffamily JP#1}}					
\newcommand{\Set}[2]{\left\{#1\,\middle|\,#2\right\}} 
\newcommand{\set}[2]{\{#1\,|\,#2\}}								  

\newcommand{\half}{{\nicefrac{1}{2}}}			

\renewcommand{\ss}{\textup{ss}}
\newcommand{\sA}{\mathcal A}

\newcommand{\sE}{\mathcal E}
\newcommand{\sN}{\mathcal N}
\newcommand{\sO}{\mathcal O}
\newcommand{\sU}{\mathcal U}
\newcommand{\sV}{\mathcal V}
\renewcommand{\sP}{\mathcal P}

\newcommand{\del}{\partial}
\newcommand{\delbar}{\bar\partial}
\newcommand{\iCR}{{\bar D}}

\tikzset{inner sep=0pt, 
  root/.style={circle,draw,minimum size=5pt,thick}, 
  cross/.style={cross out,draw,minimum size=4pt,thick},
  doubleline/.style={double distance=1.5pt,thick},
} 

\newcommand{\scale}{.3}

\newcommand{\DynkinA}{
	\begin{tikzpicture}[baseline={([yshift=-3pt]current bounding box.center)},scale=\scale]
		\node[root] (a1) at (0,0) {};
		\node[root] (a2) at (3,0) {};
		\node[cross] (a3) at (6,0) {};
		\node[root] (a4) at (9,0) {};
		\node[root] (a5) at (12,0) {};
		\draw[thick, dashed] (a1)--(a2);
		\draw[thick] (a2)--(a3);
		\draw[thick] (a3)--(a4);
		\draw[thick, dashed] (a4)--(a5);
	\end{tikzpicture}}
	
\newcommand{\DynkinD}{
	\begin{tikzpicture}[baseline={([yshift=-3pt]current bounding box.center)},scale=\scale]
		\node[root] (a1) at (0,0) {};
		\node[root] (a2) at (3,0) {};
		\node[root] (a3) at (6,0) {};
		\node[root] (a4) at (9,0) {};
		\node[root] (a5) at (12,1) {};
		\node[cross] (a6) at (12,-1) {};
		\draw[thick] (a1)--(a2);
		\draw[thick,dashed] (a2)--(a3);
		\draw[thick] (a3)--(a4);
		\draw[thick] (a4)--(a5);
		\draw[thick] (a4)--(a6);
	\end{tikzpicture}}

\newcommand{\DynkinBD}{
	\begin{tikzpicture}[baseline={([yshift=-3pt]current bounding box.center)},scale=\scale]
		\node[cross] (a1) at (0,0) {};
		\node[root] (a2) at (3,0) {};
		\node[root] (a3) at (6,0) {};
		\node[root] (a4) at (9,0) {};
		\node[root] (a5) at (12,1) {};
		\node[root] (a6) at (12,-1) {};
		\draw[thick] (a1)--(a2);
		\draw[thick,dashed] (a2)--(a3);
		\draw[thick] (a3)--(a4);
		\draw[thick] (a4)--(a5);
		\draw[thick] (a4)--(a6);
	\end{tikzpicture}}

\newcommand{\DynkinC}{
	\begin{tikzpicture}[baseline={([yshift=-3pt]current bounding box.center)},scale=\scale]
		\node[root] (a1) at (0,0) {};
		\node[root] (a2) at (3,0) {};
		\node[root] (a3) at (6,0) {};
		\node[root] (a4) at (9,0) {};
		\node[cross] (a5) at (12,0) {};
		\draw[thick] (a1)--(a2);
		\draw[thick, dashed] (a2)--(a3);
		\draw[thick] (a3)--(a4);
		\draw[doubleline] (a4)--(a5) node[pos=0.4, circle, fill, minimum size=1pt] (A){};
		\draw[thick, cap = round] (A)-- +(40:0.6);
		\draw[thick, cap = round] (A)-- +(-40:0.6);
	\end{tikzpicture}}

\newcommand{\DynkinB}{
	\begin{tikzpicture}[baseline={([yshift=-3pt]current bounding box.center)},scale=\scale]	
		\node[cross] (a1) at (0,0) {};
		\node[root] (a2) at (3,0) {};
		\node[root] (a3) at (6,0) {};
		\node[root] (a4) at (9,0) {};
		\node[root] (a5) at (12,0) {};
		\draw[thick] (a1)--(a2);
		\draw[thick, dashed] (a2)--(a3);
		\draw[thick] (a3)--(a4);
		\draw[doubleline] (a4)--(a5) node[pos=0.6, circle, fill, minimum size=1pt] (A){};
		\draw[thick, cap = round] (A)-- +(140:0.6);
		\draw[thick, cap = round] (A)-- +(-140:0.6);
	\end{tikzpicture}}

\newcommand{\DynkinEVI}{
	\begin{tikzpicture}[baseline={([yshift=-3pt]current bounding box.center)},scale=\scale]
		\node[cross] (a1) at (0,0) {};
		\node[root] (a2) at (3,0) {};
		\node[root] (a3) at (6,0) {};
		\node[root] (a4) at (6,2) {};
		\node[root] (a5) at (9,0) {};
		\node[root] (a6) at (12,0) {};
		\draw[thick] (a1)--(a2);
		\draw[thick] (a2)--(a3);
		\draw[thick] (a3)--(a4);
		\draw[thick] (a3)--(a5);
		\draw[thick] (a5)--(a6);
	\end{tikzpicture}}
	
\newcommand{\DynkinEVII}{
	\begin{tikzpicture}[baseline={([yshift=-3pt]current bounding box.center)},scale=\scale]
		\node[cross] (a1) at (0,0) {};
		\node[root] (a2) at (3,0) {};
		\node[root] (a3) at (6,0) {};
		\node[root] (a4) at (9,0) {};
		\node[root] (a5) at (9,2) {};
		\node[root] (a6) at (12,0) {};
		\node[root] (a7) at (15,0) {};
		\draw[thick] (a1)--(a2);
		\draw[thick] (a2)--(a3);
		\draw[thick] (a3)--(a4);
		\draw[thick] (a4)--(a5);
		\draw[thick] (a4)--(a6);
		\draw[thick] (a6)--(a7);
	\end{tikzpicture}}

\title[Nearly holomorphic sections]
{Nearly holomorphic sections\\on compact Hermitian symmetric spaces}
\author{Benjamin Schwarz} 


\address{Benjamin Schwarz, Universit\"{a}t Paderborn,
Fakult\"{a}t f\"{u}r Elektrotechnik, Informatik und Mathematik,
Institut f\"{u}r Mathematik, Warburger Str. 100,
33098 Paderborn, Germany}
\email{bschwarz@math.upb.de}

\begin{document}
\begin{abstract}
Let $X$ be a Kähler manifold, and $\sE$ be a Hermitian vector bundle on $X$. We investigate the space $\sN(X,\sE)$ of nearly holomorphic sections in $\sE$, which generalizes the notion of nearly holomorphic functions introduced by Shimura. If $X=U/K$ is a compact Hermitian symmetric space, and $\sE$ is $U$-homogeneous, it turns out that $\sN(X,\sE)$ coincides with the space of $U$-finite vectors in $C^\infty(X,\sE)$, and we obtain new results on the $U$-type decomposition of the Hilbert space of square integrable sections. As an application, we determine this decomposition for the holomorphic tangent space of $X$.
\end{abstract}

\maketitle


\section*{Introduction}
Let $(X,h)$ be an $n$-dimensional Kähler manifold, and $\sE$ be a holomorphic vector bundle on $X$. Let $C(X,\sE)$, $C^\infty(X,\sE)$ and $\sO(X,\sE)$ denote the space of continuous, smooth and holmomorphic sections in $\sE$. The notion of nearly holomorphic functions on $X$ was introduced by Shimura \cite{S86} in order to give a uniform description of certain automorphic forms on bounded symmetric domains. In its straightforward generalization to vector bundles, a smooth section $f\in C^\infty(X,\sE)$ is \emph{nearly holomorphic}, if it is locally given as a polynomial in $q_\ell(z):=\partial\Psi/\partial z_\ell$ ($\ell=1,\ldots, n$) with holomorphic coefficients, where $\Psi$ is a Kähler potential for $h$ on $\mathcal U\subseteq X$, and $z_1,\ldots, z_n$ are local (complex) coordinates on $\mathcal U$, i.e., for all $z\in\mathcal U$,
\begin{align}\label{eq:localterm}
	f(z) = \sum_{\b i\in\NN^n} f_\b i(z)\,q(z)^\b i
	\quad\text{with}\quad
	f_\b i\in\sO(\mathcal U,\sE_{\mathcal U})
\end{align}
and $f_\b i=0$ for almost all $\b i$. Here, we use the usual multi-index notation $q(z)^\b i=\prod_\ell q_\ell(z)^{i_\ell}$. Equivalently, $f$ is annihilated by a certain differential operator of order $m$, 
\[
	\iCR^m:C^\infty(X,\sE)\to C^\infty(X,\sE\otimes S_mT^{(1,0)}),
\]
which is called the \emph{$m$'th (iterated) invariant Cauchy--Riemann operator} (following \cite{EP96, PZ98, Zh02}), see Section~\ref{sec:InvCauchyRiem} for details. By definition,
\[
	\sN(X,\sE) := \bigcup_{m\geq 0}\sN^m(X,\sE)
	\quad\text{with}\quad
	\sN^m(X,\sE):=\ker\iCR^{m+1}\,.
\]
We note that $\sN^0(X,\sE) = \sO(X,\sE)$, the space of holomorphic sections, so
\[
	\sO(X,\sE)\subseteq\sN(X,\sE)\subseteq C^\infty(X,\sE).
\]
We propose the space of nearly holomorphic sections as the appropriate space for the discussion of questions concerning geometric analysis on Kähler manifolds. Whereas holomorphy is too restrictive and smoothness is too wide (loosing all information about the geometry), it seems that the definition of nearly holomorphic sections keeps the balance between rigidity and flexibility. Indeed, on the one hand, we show that the identity theorem holds for nearly holomorphic sections (Theorem~\ref{thm:IdTheorem}), and on the other hand, restricting to compact Hermitian symmetric spaces, we prove 

\begin{theoremAlpha}[see Theorem~\ref{thm:NearlyHolDens}]\label{thm:A}\ \\
	Let $X=U/K$ be a Hermitian symmetric space of compact type, and let $\sE=U\times_K E$ be an 
	irreducible $U$-homogeneous holomorphic vector bundle on $X$. If $\sN(X,\sE)$ is non-trivial, 
	then $\sN(X,\sE)$ is dense in $C(X,\sE)$ with respect to uniform convergence.
\end{theoremAlpha}

In fact, in our subsequent paper \cite{S12} we show that in this setting, $\sN(X,\sE)$ is non-trivial for all $U$-homogeneous holomorphic vector bundles, hence $\sN(X,\sE)$ is always dense in $C(X,\sE)$.

Even though we are able to prove that the decomposition \eqref{eq:localterm} of a nearly holomorphic section $f\in\sN(X,\sE)$ is uniquely determined by the chosen Kähler potential (see Proposition~\ref{prop:LocalDescription}), it is not clear how to determine the corresponding coefficient sections $f_\b i$. For compact Hermitian symmetric spaces, we solve this problem by using a particular Kähler potential, and the solution is given by a \emph{generalized Taylor expansion} formula, see Section~\ref{sec:taylor} for details.\\

For a first application of nearly holomorphic sections to harmonic analysis on $X=U/K$, we fix a $U$-invariant Hermitian structure on $\sE = U\times_K E$. Let $\mf u$ be the Lie algebra of $U$, and fix a Borel subalgebra $\mf b\subseteq u_\CC$ of the complexified Lie algebra with corresponding Cartan subalgebra $\mf h\subseteq\mf b$. Let $\Lambda\subseteq\mf h^*$ parametrize the isomorphism classes of irreducible representations of $U$ by their highest weights (with respect to $\mf b$), and let $V_\lambda$ denote a representative for $\lambda\in\Lambda$. Then, the space of $L^2$-sections in $\sE$ decomposes under the action of $U$ into a Hilbert sum
\begin{align*}
	L^2(X,\sE) = \widehat{\bigoplus_{\lambda\in\Lambda}}\; W_\lambda^\sE
\end{align*}
of $U$-isotypic components $W_\lambda^\sE\cong m_\lambda^\sE\cdot V_\lambda$, where $m_\lambda^\sE\geq 0$ denotes the multiplicity of $V_\lambda$ in $L^2(X,\sE)$. One of the fundamental problems in harmonic analysis is to determine the multiplicities $m_\lambda^\sE$ in an explicit manner. Frobenius reciprocity yields that
\begin{align}\label{eq:FrobRecip}
	\Hom_U(V_\lambda, L^2(X,\sE))\cong\Hom_K(V_\lambda,E)\,,
\end{align}
where on the right hand side $V_\lambda$ and $E$ are considered as $K$-modules. In particular, this shows that $m_\lambda^\sE=\dim\Hom_U(V_\lambda, L^2(X,\sE))$ is finite, and $L^2(X,\sE)$ contains precisely those $U$-types, which itself contain $E$ as a $K$-type. Classically, this correspondence is used to determine the multiplicities $m_\lambda^\sE$ in special cases. For the trivial line bundle, the Cartan--Helgason theorem gives an explicit characterization of the highest weights with positive multiplicity (spherical representations), and multiplicity freeness (i.e., $m_\lambda^\sE\leq 1$ for all $\lambda$) is obtained by an investigation of spherical vectors \cite{He84}. For general line bundles, Schlichtkrull \cite{Sch84} obtains a generalization of the Cartan--Helgason theorem, which also proves multiplicity freeness. See also \cite{Sh94}. In the case of higher rank vector bundles, little seems to be known. Even though a Cartan--Helgason like theorem for arbitrary $K$-types is proved by Kostant \cite{Kos04}, it remains a non-trivial problem to derive the $U$-type decomposition of $L^2(X,\sE)$ from this result. In \cite{Ca05b, Ca05a}, Camporesi applies this approach to compact Riemannian symmetric spaces of rank one (certainly including Hermitian symmetric spaces of rank one), and obtains a general description of the set of highest weights occurring in $L^2(X,\sE)$, but without determining the precise multiplicities.\\

Instead of using Frobenius reciprocity, we prove a new correspondence between $U$-types in $L^2(X,\sE)$ and $K$-types in a certain $K$-module by an investigation of nearly holomorphic sections. In order to formulate our result, recall that $\mf u_\CC$ admits the grading $\mf u_\CC = \mf n^+\oplus\mf k_\CC\oplus\mf n^-$, where $\mf n^\pm$ are $K$-invariant abelian subalgebras, and $\mf k$ is the Lie algebra of $K\subseteq U$. We may assume that the Borel subalgebra $\mf b\subseteq\mf u_\CC$ is chosen such that $\mf n^+\subseteq\mf b$, and $\mf b':=\mf b\cap\mf k_\CC$ is a Borel subalgebra of $\mf k_\CC$. Hence, $\mf h\subseteq\mf k_\CC$ and $\mf h$ is also a Cartan subalgebra of $\mf k_\CC$. Recall that for all $\lambda\in\Lambda$, the action of $K$ on the subspace $V_\lambda^{\mf n^+}\subseteq V_\lambda$ of $\mf n^+$-invariants,
\[
	V_\lambda^{\mf n^+} := \Set{v\in V_\lambda}{Y.v = 0\text{ for all }Y\in\mf n^+},
\]
is irreducible with highest weight $\lambda$. Let $S\mf n^+$ denote the symmetric algebra of $\mf n^+$, equipped with the adjoint action of $K$.

\begin{samepage}
\begin{theoremAlpha}[see Theorem~\ref{thm:highestweights}]\label{thm:B}\ \\
	Let $X=U/K$ be an irreducible Hermitian symmetric space of compact type, and let 
	$\sE=U\times_KE$ be an irreducible $U$-homogeneous Hermitian vector bundle. If 
	$\sO(X,\sE)\neq\{0\}$, then
	\begin{align}\label{eq:UKisomorphism}
		\Hom_U(V_\lambda,L^2(X,\sE))\cong\Hom_K(V_\lambda^{\mf n^+},E\otimes S\mf n^+)
	\end{align}
	for all $\lambda\in\Lambda$. Moreover, any $K$-type in $E\otimes S\mf n^+$ is isomorphic to 
	$V_\lambda^{\mf n^+}$ for some $\lambda\in\Lambda$,	hence there is a bijection between $U$-types 
	in $L^2(X,\sE)$ and $K$-types in $E\otimes S\mf n^+$.
\end{theoremAlpha}
\end{samepage}

We note that due to the Borel-Weil Theorem, $\sO(X,\sE)$ is non-trivial, if and only if the $K$-type $E$ is isomorphic to $V_\lambda^{\mf n^+}$ for some $\lambda\in\Lambda$. This shows that the non-triviality of $\sO(X,\sE)$ is essential for the result of Theorem~\ref{thm:B}. However, for more general vector bundles, it is still true that $\Hom_U(V_\lambda,L^2(X,\sE))$ embeds into $\Hom_K(V_\lambda^{\mf n^+},E\otimes S\mf n^+)$ for all $\lambda\in\Lambda$, see Theorem~\ref{thm:highestweights}.

The basic idea underlying Theorem~\ref{thm:B} is to consider the isomorphism
\[
	\Hom_U(V_\lambda,L^2(X,\sE))\stackrel{\cong}{\longrightarrow}\Hom_K(V_\lambda^{\mf n^+},L^2(X,\sE)^{\mf n^+}),\
	T\mapsto T\at{V_\lambda^{\mf n^+}},
\]
and to replace $L^2(X,\sE)^{\mf n^+}$ on the right hand side by a more accessible representation of $K$. Theorem~\ref{thm:B} states, that if $\sE$ admits non-trivial holomorphic sections, then $E\otimes S\mf n^+$ is the right replacement. This is motivated by two crucial observations. Firstly, as a consequence of Theorem~\ref{thm:A}, the space of nearly holomorphic sections coincides with the space of $U$-finite vectors in $L^2(X,\sE)$ (see Proposition~\ref{prop:NearlyHolomorphicAndUFinite}), i.e.,
\begin{align*}
	\sN(X,\sE) = \bigoplus_{\lambda\in\Lambda} W_\lambda^\sE.
\end{align*}
Therefore, all sections in $L^2(X,\sE)$ that are of interest for the discussion of the $U$-type decomposition are nearly holomorphic. Secondly, considering the expansion of nearly holomorphic sections as in \eqref{eq:localterm} on $\mf n^+\subseteq X$ (holomorphically embedded via the exponential map), it turns out that $\mf n^+$-invariant sections correspond to expansions with \emph{constant} coefficient sections $f_\b i$ (if the Kähler potential is chosen properly). This yields a $K$-equivariant embedding of $\mf n^+$-invariant nearly holomorphic sections into $E\otimes S\mf n^+$. A detailed analysis of nearly holomorphic functions then shows that this is in fact an isomorphism of $K$-modules, see Corollary~\ref{cor:nearlyholomorphicgenerators} which is essential for this part of the proof of Theorem~\ref{thm:B}.\\

We list some applications of Theorem~\ref{thm:B}.
\begin{enumerate}[(i)]
	\item For the trivial line bundle, Theorem~\ref{thm:B} states that the Cartan--Helgason theorem	
				(applied to Hermitian symmetric spaces of compact type) is equivalent to the decomposition 
				of the symmetric algebra $S\mf n^+$ under the action of $K$, which is 
				well-known due to the work of Hua (classical, \cite{Hu63}), Kostant (unpublished), and 
				Schmid \cite{Sc69}. In fact, this equivalence has been the original motivation for our 
				investigation. Since the isomorphism \eqref{eq:UKisomorphism} is explicitly given, we are 
				able to translate known formulas for highest weight vectors in $S\mf n^+$ (due to Upmeier 
				\cite{Up86}) to explicit formulas for highest weight vectors in $L^2(X)$, see 
				Remark~\ref{rmk:ExplicitHighestWeights}.
	\item For general line bundles which admit holomorphic sections, Theorem~\ref{thm:B}
				recovers Schlichtkrull's result \cite{Sch84} generalizing the Cartan--Helgason theorem, 
				see Remark~\ref{rmk:linebundles}.
	\item To demonstrate the use of our result for vector bundles of higher rank, we  
				obtain the precise decomposition of the holomorphic tangent bundle. We note that this 
				decomposition is not multiplicity free (except for rank 1), see
				Theorem~\ref{thm:tangentdecomp}.
	\item For general vector bundles, an immediate consequence of Theorem~\ref{thm:B} is that the 
				multiplicities $m_\lambda^\sE$ are uniformly bounded by the dimension of the fiber 
				$E$, see Corollary~\ref{cor:weightdescription}. In the setting of bounded symmetric 
				domains, similar results have been proved by T.~Kobayashi \cite{Ko08}.
\end{enumerate}

\subsection*{Further remarks}
Even though nearly holomorphic sections are defined on any Kähler manifold, they have been studied in detail only in the Hermitian symmetric case (to the best of our knowledge), and even in this setting, almost all results are concerned with line bundles on bounded symmetric domains (having a global Kähler potential), see \cite{S86,S87,Zh02}. Our result applied to line bundles can be considered as dual to the results obtained by Zhang in \cite{Zh02}, where nearly holomorphic functions are used to describe relative discrete series of weighted $L^2$-spaces on bounded symmetric domains. However, we note that since compact Hermitian symmetric spaces do not admit global Kähler potentials, it is non-trivial to transfer local to global properties.

In \cite{PZ93}, Peetre--Zhang cover the case of line bundles on the Riemann sphere, and in \cite{S87} Shimura obtains for the classical compact Hermitian symmetric spaces (by a case-by-case analysis) explicit formulas for generators of the algebra of nearly holomorphic functions. Using Jordan theory, we obtain a uniform description of these generators which also applies to the exceptional Hermitian symmetric spaces, see Corollary~\ref{cor:nearlyholomorphicgenerators}. The iterated invariant Cauchy--Riemann operators are discussed in more generality by Engli\v{s} and Peetre in \cite{EP96}, however we note that their results are not used in our proofs.

\subsection*{Outlook}
Both of our main results so far rely on the existence of non-trivial (nearly) holomorphic sections. In our subsequent paper \cite{S12} we show that in case of Theorem~\ref{thm:A}, this assumption can be dropped, since the space of nearly holomorphic sections turns out to be non-trivial for all $U$-homogeneous vector bundles. Therefore, $\sN(X,\sE)$ is always dense in $C(X,\sE)$ with respect to uniform convergence. We also improve Theorem~\ref{thm:B} by characterizing a $K$-invariant subspace $S(E)\subseteq E\otimes S\mf n^+$ such that Theorem~B holds for all $U$-homogeneous vector bundles if $E\otimes S\mf n^+$ is replaced by $S(E)$. Both results are obtained by solving the extension problem for nearly holomorphic sections, i.e., by characterizing those local nearly holomorphic sections $f\in\sN(\mf n^+,\sE_{\mf n^+})$ (with $\mf n^+$ considered as an open and dense subset of $X$) which extend to global nearly holomorphic sections.

An open problem is the generalization of our results to non-symmetric Kähler manifolds. As a first step, one might consider generalized flag varieties $X=G/Q$ where $G$ is the same complex Lie group as above and $Q$ is any parabolic subgroup of $G$. In this case, the grading of $\mf u_\CC$ corresponding to $Q$ becomes more complicated, and Jordan theoretic arguments used in our investigation must be replaced by corresponding Lie theoretic arguments. We note that beyond the Hermitian symmetric case there are different $U$-invariant Kähler metrics defined on $X$, and the Killing metric is not Kähler in general. However, there is a unique Kähler--Einstein metric, see e.g.\ \cite{A06}.

\subsection*{Organisation}
This paper is organized as follows. In Section~\ref{sec:NearlyHolSections} recall basic properties of the invariant Cauchy--Riemann operator for holomorphic vector bundles on general Kähler manifolds and introduce the notion of nearly holomorphic sections. We further investigate the local characterization of these sections and prove the identity theorem as well as the finite dimensionality of $\sN^m(X,\sE)$ on compact Kähler manifolds. Starting with Section~\ref{sec:NearHolOnHermSym}, we confine to $U$-homogeneous Hermitian vector bundles on Hermitian symmetric spaces of compact type, $X=U/K$. Here, we prove Theorem~\ref{thm:A} and explicitly describe generators for the algebra of nearly holomorphic functions. We also recall the notions from Jordan theory used in the proofs. Section~\ref{sec:RepTheoApp} is devoted to the discussion of the $U$-type decomposition of $L^2(X,\sE)$. We first link nearly holomorphic sections to this problem by showing that $\sN(X,\sE)$ coincides with the space of $U$-finite vectors in $L^2(X,\sE)$ (assuming $\sN(X,\sE)$ is non-trivial). Then, Theorem~\ref{thm:B} and various consequences are proved. In particular, the decomposition problem is solved for the case of the holomorphic tangent bundle. In Section~\ref{sec:taylor}, we prove a generalized Taylor expansion formula for nearly holomorphic sections. For convenience to the reader, the appendix provides some data of the classification of Hermitian symmetric spaces.\\[2mm]
\noindent
\emph{Acknowledgment:} I would like to thank Joachim Hilgert, Jan Möllers, and Henrik Seppänen for helpful discussions on the topic of this paper. I am also grateful to Bent \char216rsted for bringing Zhang's paper \cite{Zh02} to my attention, and hence initiating this work on nearly holomorphic sections.
\newpage

\section{Invariant Cauchy--Riemann operators and nearly holomorphic sections}\label{sec:NearlyHolSections}
Throughout this section, let $(X,h)$ be a (possibly non-compact) Kähler manifold of complex dimension $n$. For $x\in X$, the complexification of the real tangent space $T_x$ is denoted by $T_x^\CC$, and its decomposition into the holomorphic and the anti-holomorphic part is $T_x^\CC = T_x^{(1,0)}\oplus T_x^{(0,1)}$. The Kähler metric, $h$,  induces a smooth bundle isomorphism of the antiholomorphic cotangent bundle and the holomorphic tangent bundle, denoted by
\[
 \eta:(T^{(0,1)})^*\stackrel{\cong}{\longrightarrow} T^{(1,0)}. 
\]
Let $\sE$ be a holomorphic vector bundle on $X$, and let $C(X,\sE)$, $C^\infty(X,\sE)$, and $\sO(X,\sE)$, denote the space of continuous, smooth, and holomorphic sections in $\sE$, respectively. If, in addition, $H$ is an Hermitian structure on $\sE$, let $L^2(X,\sE)$ denote the Hilbert space of square integrable sections. Let
\[
	\delbar: C^\infty(X,\sE)\to C^\infty(X,\sE\otimes (T^{(1,0)})^*)
\]
denote the usual Dolbeault operator of $\sE$ defined by trivializations.

\subsection{Invariant Cauchy--Riemann operator}\label{sec:InvCauchyRiem}
The \emph{invariant Cauchy--Riemann operator} is defined by (cf.\ \cite{EP96})
\[
	\iCR_\sE := \eta^*\circ\delbar: C^\infty(X,\sE)\to C^\infty(X,\sE\otimes T^{(1,0)})\;,
\]
where $\eta^*:C^\infty(X,\sE\otimes (T^{(0,1)})^*)\to C^\infty(X,\sE\otimes T^{(1,0)})$ denotes the isomorphism induced by $\eta$. If $(z^1,\ldots, z^n)$ are coordinates on an open subset $\sU\subseteq X$, and $(e_\alpha)_{\alpha=1,\ldots,N}$ is a system of local trivializing holomorphic sections of $\sE_\sU$ ($=\pi^{-1}(\sU)$ where $\pi:\sE\to X$ denotes the canonical projection), then $\iCR_\sE$ is locally given by
\begin{align}\label{eq:localiCR}
	\iCR_\sE(\sum_\alpha f^\alpha e_\alpha)
	= \sum_{i,j,\alpha}h^{\bar j i}
		\frac{\partial f^\alpha}{\partial\bar z^j}\,e_\alpha\otimes\del_i\,,
\end{align}
where $\del_i := \tfrac{\partial}{\partial z^i}$ is the standard basis of the holomorphic tangent space, and $h^{\bar j i}$ is the matrix inverse of the Kähler metric coefficients $h_{i\bar j}:=h(\del_i,\delbar_j)$.

Since $\sE\otimes T^{(1,0)}$ is again a holomorphic vector bundle on $X$, the invariant Cauchy--Riemann operator can be iterated, i.e., we set $\iCR_\sE^1 := \iCR_\sE$ and define inductively
\[
	\iCR_\sE^m:=\iCR_{\sE\otimes (T^{(1,0)})^{\otimes (m-1)}}\circ\iCR_\sE^{m-1}
\]
for $m>1$. By abuse of notation, we omit the respective index referring to the vector bundle and simply write
\[
	\iCR^m =\iCR\circ\cdots\circ\iCR:
	C^\infty(X,\sE)\to C^\infty(X,\sE\otimes(T^{(1,0)})^{\otimes m})\,.
\]
We recall some basic properties.

\begin{samepage}
\begin{proposition}[\cite{EP96,PZ98}]\label{prop:BasisPropsNearHol}\
	\begin{enumerate}
		\item \emph{Invariance}.\\
					If $g$ is a metric preserving biholomorphic map on $X$ that lifts to a holomorphic 
					bundle isomorphism $\tilde g$ on $\sE$, then
					\[
						\iCR(\tilde g^* f) = (\tilde g^*\otimes dg^*)\iCR f\;,
					\]
					where $\tilde g^*$ and $dg^*$ denote the induced actions on sections of 
					$\sE$ and $T^{(1,0)}$, respectively.\vspace{2mm}
		\item \emph{Symmetry}.\\
					The iterate $\iCR^m$ of $\iCR$ maps $C^\infty(X,\sE)$ into 
					$C^\infty(X,\sE\otimes S_mT^{(1,0)})$, where
					$S_mT^{(1,0)}\subseteq (T^{(1,0)})^{\otimes m}$ denotes the subbundle of symmetric 
					tensors.
	\end{enumerate}
\end{proposition}
\end{samepage}
We note that the symmetry property also follows from Lemma~2.0 in Shimuras paper \cite{S86}, in which the concept of nearly holomorphic functions is introduced for the first time.

\begin{remark}
	The invariant Cauchy--Riemann operator can be defined in the same way for the more general class 
	of Hermitian manifolds. However, the proof of the symmetry property \cite{PZ98, S86} relies on
	the Kähler property of the metric.
\end{remark}

\begin{remark}\label{rmk:ellipticoperators}
	We like to point out that the sequence of operators
	\[
		C^\infty(X,\sE)\stackrel{\iCR}{\rightarrow} C^\infty(X,\sE\otimes T^{(1,0)})
		\stackrel{\iCR}{\rightarrow}\cdots\stackrel{\iCR}{\rightarrow}
		C^\infty(X,\sE\otimes S_mT^{(1,0)})\stackrel{\iCR}{\rightarrow}\cdots
	\]
	is not a complex. In fact, the non-triviality of the iterated Cauchy--Riemann operators $\iCR^m$ 
	is the core of the theory of nearly holomorphic sections, which we discuss in the next section. 
	Nevertheless there are elliptic operators, $L_m$, attached to each of the 
	operators $\iCR^m$, which have been introduced by Peetre, Peng and Zhang \cite{PPZ90, PZ93} on 
	the unit disc, and which are discussed in full generality in \cite{EP96}. For the definition of 
	$L_m$, choose some Hermitian structure $H$ on $\sE$, extend this in combination with the 
	Kähler metric $h$ to Hermitian structures on the vector bundles $\sE\otimes S_m T^{(1,0)}$, 
	$m\geq 0$, and let $(\iCR^m)^*$ be the adjoint of $\iCR^m$ with respect to the corresponding 
	Hermitian inner products. Then, the \emph{generalized Laplacians} $L_m$ ($m\geq 1$) are 
	defined by
	\[
		L_m:= (\iCR^m)^*\circ\iCR^m: C^\infty(X,\sE)\to C^\infty(X,\sE)\,.
	\]
	Indeed, for the trivial line bundle $\sE=X\times\CC$ with standard Hermitian structure, 
	$L_1$ coincides with the familiar Laplace--Beltrami operator on $X$, see \cite[p.23]{EP96}. In 
	general, we claim that each $L_m$ is an elliptic operator (which has not been noted in 
	\cite{EP96, PPZ90, PZ93}). For the notation and details of elliptic operator theory we refer to 
	\cite{We08}. We first note that the definitions immediately imply that $\iCR^m$ is a 
	differential operator of order $m$, hence $(\iCR^m)^*$ and $L_m$ are also differential operators 
	of order $m$ and $2m$, respectively. Moreover, it follows from \eqref{eq:localiCR} that the 
	symbol of $\iCR$ is given by
	\[
		\sigma(\iCR)(x,\xi):\sE_x\to\sE_x\otimes T_x^{(1,0)},\
		e\mapsto e\otimes i\,\widehat{\xi^{(0,1)}}\,,
	\]
	where $\xi^{(0,1)}$ denotes the projection of $\xi\in T_x^*$ onto the dual antiholomorphic 
	tangent space $(T_x^{(0,1)})^*$, and $\widehat{\xi^{(0,1)}}$ is the holomorphic tangent vector
	dual to $\xi^{(0,1)}$ with respect to the Kähler metric $h$. Now, the symbol of the 
	iterated operator $\iCR^m$ is
	\[
		\sigma(\iCR^m)(x,\xi):\sE_x\to\sE_x\otimes S_mT_x^{(1,0)},\
		e\mapsto e\otimes (i\,\widehat{\xi^{(0,1)}})^m\,,
	\]
	where $\zeta^m$ denotes the $m$'th symmetric tensor product of $\zeta\in T_x^{(1,0)}$. Since the 
	symbol of the adjoint operator is the adjoint of the symbol, it follows that
	\[
		\sigma((\iCR^m)^*)(x,\xi):\sE_x\otimes S_mT_x^{(1,0)}\to\sE_x,\
		e\otimes \prod_j v_j\mapsto e\cdot\prod_j(-i\,\xi^{(1,0)}(v_j))
	\]
	(linearly extended to all of $\sE_x\otimes S_mT_x^{(1,0)}$). Finally, this yields that the 
	symbol of $L_m$ is given by
	\[
		\sigma(L_m)(x,\xi):\sE_x\to\sE_x,\
		e \mapsto e\cdot h_x^*(\xi,\xi)^m\,,
	\]
	where $h^*$ denotes the Hermitian inner product on the cotangent bundle $T^*$. This is an 
	isomorphism for all $\xi\neq 0$, and hence $L_m$ is an elliptic operator.
\end{remark}

\subsection{Nearly holomorphic sections}\label{subsec:NearlyHolSections}
A section $f\in C^\infty(X,\sE)$ is called \emph{nearly holomorphic}, if it is annihilated by $\iCR^m$ for some $m\in\NN$. For the smallest such $m$, we call $\deg f:=m-1$ the \emph{degree} of $f$, and denote by
\[
	\sN^m(X,\sE) := \ker\iCR^{m+1}
\]
the space of all nearly holomorphic sections of degree at most $m$. For convenience we also set $\sN^{-1}(X,\sE):=\{0\}$. The space of \emph{all} nearly holomorphic sections is denoted by
\[
	\sN(X,\sE) := \bigcup_{m=0}^\infty \sN^m(X,\sE)\;.
\]
For the trivial line bundle $\sE = X\times\CC$, we set $\sN^m(X):=\sN^m(X,X\times\CC)$, $\sN(X):=\sN(X,X\times\CC)$. These are the \emph{nearly holomorphic functions} introduced by Shimura \cite{S86}.

\begin{remark}
	Since $\iCR^m$ is a differential operator, we may also consider its action on local sections 
	$f\in C^\infty(\sU,\sE_\sU)$ for $\sU\subseteq X$ open. Then, $\sN(\sU,\sE_\sU)$ is defined in 
	the same way as for global sections. It readily follows that the assignment
	$\sU\mapsto\sN(\sU,\sE_\sU)$ in combination with the natural restriction maps form an abelian 
	sheaf.
\end{remark}

The term 'nearly holomorphic' is motivated by the following local characterization. Let $\sU\subseteq X$ be an open subset with local coordinates $z=(z^1,\ldots, z^n)$ and a Kähler potential $\Psi\in C^\infty(\sU)$. Then define 
\begin{align}\label{eq:qmapcomponents}
	q:=\del\Psi\in C^\infty(\sU,(T^{(1,0)}_\sU)^*)\,,\quad
	q_\ell:=\frac{\partial\Psi}{\partial z^\ell}\in C^\infty(\sU)\,.
\end{align}
We call $q$ the \emph{$q$-map of $\Psi$}, and note that the $q_\ell$ are the coefficient functions of $q$ with respect to the chosen coordinates. For $\b i\in\NN^n$ we use the usual multi-index notation, so $|{\b i}| = i_1+\cdots+i_n$ and $q(z)^\b i = \prod_{\ell=1}^n q_\ell(z)^{i_\ell}$. 

\begin{proposition}\label{prop:LocalDescription}
	A local section $f\in C^\infty(\sU,\sE_\sU)$ is nearly holomorphic of degree $m$ if and only 
	if it is a polynomial in $q_1,\ldots,q_n$ of degree $m$ with holomorphic coefficients, 
	i.e., if there exist local holomorphic sections $f_\b i\in\sO(\sU,\sE_\sU)$ such that
	\begin{align}\label{eq:NearlyHolPolyDecomp}
		f(z) = \sum_{|{\b i}|\leq m} f_\b i(z)\cdot q(z)^\b i
	\end{align}
	for all $z\in\sU$, and $f_\b i\neq 0$ for some $\b i\in\NN^n$ with $|\b i|=m$. Moreover, the 
	coefficient sections $f_\b i$ are uniquely determined by the choice of $\Psi$.
\end{proposition}
\begin{proof}
The main part of this proposition follows essentially from the corresponding result in the special case of the trivial line bundle, which is proved in \cite{S86}. First choose local trivializing holomorphic sections $(e_\alpha)_{\alpha = 1,\ldots,N}$ on some open subset $\sV\subseteq\sU$ and decompose $f = \sum_\alpha f^\alpha e_\alpha$. Since $\iCR$ acts trivially on the holomorphic sections $e_\alpha$, $f$ is a nearly holomorphic section if and only if each $f^\alpha$ is a nearly holomorphic function. In this case, \cite[Lemma~2.1]{S86} implies that 
\begin{align}\label{eq:NearlyHolPolyDecompCoeff}
	f^\alpha(z) = \sum_{|{\b i}|\leq m}f^\alpha_\b i(z)\cdot q^\b i(z)
\end{align}
for some holomorphic maps $f^\alpha_\b i\in\sO(\sV)$. We show that this decomposition is unique. Set $\iCR_\ell:=\sum_k h^{\bar k\ell}\tfrac{\partial}{\partial \bar z^k}$. Due to \cite[Lemma~2.0]{S86}, these operators commute, cf. also Proposition~\ref{prop:BasisPropsNearHol}\,(b). Let $\iCR^\b i:=\prod_\ell (\iCR_\ell)^{i_\ell}$. It easily follows that $\iCR_\ell q_k(z) = \delta_{\ell k}$ for all $\ell, k$, and hence
\begin{align*}
	(\iCR^\b iq^\b j)(z) =
		\begin{cases}
			\tfrac{\b j!}{(\b j-\b i)!}\,q^{\b j-\b i}(z) & \text{if $\b i\leq\b j$,}\\
			0	& \text{otherwise.}
		\end{cases}
\end{align*}
Here we use the abbreviations $\b j! := \prod j_\ell!$, $(\b i-\b j)_\ell := i_\ell-j_\ell$, and $\b i\leq\b j$ if and only if $i_\ell\leq j_\ell$ for all $\ell$. In particular,
\begin{align*}
	(\iCR^\b iq^\b j)(z)
		= \b i!\cdot\delta_{i_1,j_1}\cdots\delta_{i_n,j_n}
	\quad\text{if}\quad |{\b i}|\geq|{\b j}|\,,
\end{align*}
This is the local version of \cite[Lemma~2.2]{Zh02}. The uniqueness statement now follows by induction on $m$. For $m=0$, there is nothing to prove. For $m>0$, applying $\iCR^\b i$ with $|{\b i}| = m$ to \eqref{eq:NearlyHolPolyDecompCoeff} yields $\iCR^\b if(z) =f_\b i^\alpha(z)$, thus the holomorphic coefficient functions $c_\b i^\alpha$ with $|{\b i}| = m$ are unique. Subtracting these terms from $f^\alpha$ yields a local nearly holomorphic function of degree $< m$, so by induction hypothesis, the coefficient functions of the lower degrees are also uniquely determined by $f$.\\
Now suppose that $(\tilde e_\alpha)_{\alpha = 1,\ldots,N}$ is a different system of local trivializing holomorphic sections on $\sV$, and let $\tilde f^\alpha$ and $\tilde f^\alpha_\b i$ be the corresponding coefficient functions. Then $f^\alpha(z) = \sum_\beta T^\alpha_\beta(z)\tilde f^\beta(z)$ for some holomorphic maps $T^\alpha_\beta\in\sO(\sV)$. The uniqueness of the coefficients in \eqref{eq:NearlyHolPolyDecompCoeff} now implies $f^\alpha_\b i = \sum_\beta T^\alpha_\beta\tilde f^\beta_\b i$ for all $\b i$. Therefore,
\[
	f_\b i
		:=\sum_\alpha f_\b i^\alpha e_\alpha
		= \sum_\alpha \tilde f_\b i^\alpha\tilde e_\alpha
\]
is independent of the choice of the trivializing sections. Choosing an open covering of $\sU$ with local trivializing holomorphic sections, this shows that $f_\b i$ is defined on all of $\sU\subseteq X$, and it is uniquely determined by the Kähler potential $\Psi$.
\end{proof}

\begin{remark}\label{rmk:coeffQuestion}
	We like to stress that the procedure for determining the coefficient sections $f_\b i$ of 
	\eqref{eq:NearlyHolPolyDecomp} starts with the coefficients of highest degree, $|\b i| = \deg f$.
	This is somehow unsatisfactory, since this assumes the a priori knowledge of the degree 
	of $f$. Moreover, in view of expanding more general sections into a series of nearly holomorphic 
	sections (as a sort of Taylor expansion), it would be desirable to have non-recursive formulas 
	for the coefficient sections. In fact, it suffices to have a formula for the lowest degree term
	$f_\b 0$, since applying such a formula to $\iCR^{\b i}f$ (with $\iCR^{\b i}$ defined as above) 
	also yields the coefficients of the higher degree terms. For compact Hermitian symmetric spaces, 
	this problem is solved in Section~\ref{sec:taylor} by using a particular Kähler potential. 
\end{remark}

\begin{remark}\label{rmk:DifferentKaehlerPotential}
	We briefly indicate what happens if we choose a different Kähler potential $\tilde\Psi$ on 
	$\sU$. Since the difference $\Psi-\tilde\Psi$ is the real part of a holomorphic function 
	$g\in\sO(\sU)$, it follows that the maps $q_\ell=\partial_\ell\Psi$ and
	$\tilde q'_\ell=\partial_\ell\tilde\Psi$ satisfy $q_\ell=\tilde q_\ell+\partial_\ell g$ for all 
	$\ell$. Now it is a combinatorial exercise to determine from \eqref{eq:NearlyHolPolyDecomp} the 
	relation between the corresponding coefficient sections $f_\b i$ and $\tilde f_\b j$. We note 
	that the coefficient sections of highest order, $|{\b i}| = m$, are independent of the choice of 
	the Kähler potential. In fact, up to constants these are the coefficient functions of the 
	holomorphic section	$\bar D^mf\in\sO(\sU,\sE_\sU\otimes S_mT^{(1,0)})$.
\end{remark}

\begin{theorem}[Identity theorem]\label{thm:IdTheorem}
	Let $X$ be connected and $\sU\subseteq X$ be open. Then, the restriction map
	\[
		\sN(X,\sE)\to\sN(\sU,\sE_\sU),\ f\mapsto f\at{\sU}
	\] 
	is an injection.
\end{theorem}
\begin{proof}
For $f\in\sN^m(X,\sE)$ with $m$ smallest possible, assume that $f|_\sU=0$. If $m\geq 0$, then $\iCR^m f$ is a section in $\sE\otimes S_mT^{(1,0)}$, and in the kernel of $\iCR$. Since $\iCR = \eta^*\circ\delbar$ where $\eta^*$ is an isomorphism, the kernel of $\iCR$ coincides with the kernel of $\delbar$. Hence $\iCR^m f$ is a holomorphic section in $\sE\otimes S_mT^{(1,0)}$. The assumption $f|_\sU = 0$ also implies $\iCR^m f|_\sU = 0$. According to the identity theorem for holomorphic sections, this yields $\iCR^m f = 0$, and hence $f\in\sN^{m-1}(X,\sE)$, which contradicts the minimality of $m$. Therefore, $m=-1$ and $f\in\sN^{-1}(X,\sE)=\{0\}$.
\end{proof}

In connection with the identity theorem, Proposition~\ref{prop:LocalDescription} implies that a global nearly holomorphic section $f\in\sN(X,\sE)$ is uniquely determined by the corresponding set of local holomorphic sections $f_\b i\subseteq\sO(\sU,\sE_\sU)$ for $\sU\subseteq X$ open. In particular, we note that $f$ has degree $m$ if and only if $f|_\sU$ has degree $m$, independent of the choice of $\sU$. This observation also yields the following result.

\begin{corollary}\label{cor:NearHolModule}
	$\sN(X)$ is a filtered algebra,
	\[
		\sN^m(X)\cdot\sN^{m'}(X)\subseteq\sN^{m+m'}(X)\,,
	\]
	Moreover, $\sN(X,\sE)$ is a filtered $\sN(X)$-module with 
	\[
		\deg(fg) = \deg f + \deg g
	\]
	for $f\in\sN(X)$, $g\in\sN(X,\sE)$.
\end{corollary}

We close the discussion of nearly holomorphic sections on general Kähler manifolds with the following statement.

\begin{proposition}\label{prop:FiniteDim}
	If $X$ is compact with finitely many connected components, then $\sN^m(X,\sE)$ is finite 
	dimensional for all $m$.
\end{proposition}
\begin{proof}
Due to Remark~\ref{rmk:ellipticoperators}, $\sN^m(X,\sE)$ is a subspace of $\ker L_{m+1}$ for the elliptic operator $L_{m+1} = (\iCR^{m+1})^*\circ\iCR^{m+1}$. Now compactness of $X$ implies that $\ker L_{m+1}$ is finite dimensional, see \cite[Theorem~4.8 in Chapter IV]{We08}. 
\end{proof}

\begin{remark}
	We note that Proposition~\ref{prop:FiniteDim} also follows from the more familiar fact (also 
	proved by elliptic operator theory) that the space of holomorphic sections in a vector bundle of 
	a compact manifold is finite dimensional. This implies that the kernel of each operator in the 
	sequence
	\[
		C^\infty(X,\sE)\stackrel{\iCR}{\rightarrow} C^\infty(X,\sE\otimes T^{(1,0)})
		\stackrel{\iCR}{\rightarrow}\cdots\stackrel{\iCR}{\rightarrow}
		C^\infty(X,\sE\otimes S_mT^{(1,0)})\stackrel{\iCR}{\rightarrow}\cdots
	\]
	is finite dimensional (since $\ker\iCR = \ker\delbar$). Now basic arguments form linear 
	algebra show that the kernel of the iterated operator $\iCR^{m+1}$ is finite dimensional too.
\end{remark}


\section{Nearly holomorphic sections on compact\\Hermitian symmetric spaces}\label{sec:NearHolOnHermSym}
Let $X$ be an irreducible Hermitian symmetric space of compact type. Let $U$ be the simply connected covering group of the identity component of the automorphism group of $X$, which is a compact group. Then, $X=U/K$ where $K\subseteq U$ denotes the stabilizer subgroup of some base point in $X$. If $G$ denotes the complexification of $U$, which is also simply connected, we may view $X$ as complex flag manifold, $X=G/P$, with maximal parabolic subgroup $P\subseteq G$ such that $K\subseteq P$. The elements of $G$ act on $X$ by biholomorphic transformations. For details, we refer to \cite{He01}.

In the following, a $G$-homogeneous holomorphic vector bundle $\sE=G\times_PE$ on $X$ is called \emph{irreducible}, if the associated holomorphic representation $\rho:P\to\GL(E)$ is irreducible. 

From the invariance of the Cauchy--Riemann operator it follows that the space $\sN(X,\sE)$ of nearly holomorphic sections is $U$-invariant. In fact, each subspace $\sN^m(X,\sE)$ is $U$-invariant. The main goal of this section is the following result, representation theoretic applications are discussed in Section~\ref{sec:RepTheoApp}.

\begin{theorem}\label{thm:NearlyHolDens}
	Let $\sE$ be an irreducible $G$-homogeneous holomorphic vector bundle on $X$. If $\sN(X,\sE)$ is 
	non-trivial, then $\sN(X,\sE)$ is dense in $C(X,\sE)$ with respect to uniform convergence. 
\end{theorem}

\begin{remark}
	We note that in particular, Theorem~\ref{thm:NearlyHolDens} applies to $G$-homogeneous 
	holomorphic vector bundles that admit non-trivial holomorphic sections. In the successive paper 
	\cite{S12}, we improve this statement by showing that in fact every $G$-homogeneous holomorphic
	vector bundle admits non-trivial nearly holomorphic sections, hence $\sN(X,\sE)$ is	dense in 
	$C(X,\sE)$ in any case.
\end{remark}

For the proof of Theorem~\ref{thm:NearlyHolDens}, we first show that it suffices to prove that the space $\sN(X)$ of nearly holomorphic functions is dense in $C(X)$. The basic idea of the next lemma is taken from \cite[Theorem~5]{Mos06}.

\begin{lemma}\label{lem:densityCrit}
	Assume that $\sA\subseteq C(X)$ is a dense subalgebra, and $\Sigma\subseteq C(X,\sE)$ is a 
	subset such that for each $x\in X$, the fiber $\sE_x$ is spanned by 
	$\Set{\sigma(x)}{\sigma\in\Sigma}$. Then, the $\sA$-module generated by $\Sigma$ is dense in 
	$C(X,\sE)$ with respect to uniform convergence.
\end{lemma}
The proof of this lemma is straightforward, using the existence of a partition of unity corresponding to an open covering of $X$ that trivializes $\sE$. We omit the details.

In the setting of Theorem~\ref{thm:NearlyHolDens}, the action of $U$ on any non-trivial section $f\in C(X,\sE)$ generates a subset $\Sigma:=U.f\subseteq C(X,\sE)$, that satisfies the assumption of Lemma~\ref{lem:densityCrit}. Indeed, since $U$ acts transitively on $X$, there is a section $f'\in\Sigma$ with $f'(o)\neq 0$ for the base point $o=eK$. Since $\sE$ is irreducible, the $K$-action on $f'$ generates a spanning set $\Set{(k.f')(o)=\rho(k)^{-1}f'(o)}{k\in K}$ for the fiber $\sE_o$. Using again transitivity of the $U$-action on $X$, this also shows that for each $x\in X$, the fiber $\sE_x$ is spanned by $\set{f'(x)}{f'\in\Sigma}$.

Applying this argument to a non-trivial nearly holomorphic section $f\in\sN(X,\sE)$, and recalling that $\sN(X,\sE)$ is $U$-invariant, it follows that $\Sigma:=\sN(X,\sE)$ satisfies the assumption of Lemma~\ref{lem:densityCrit}. Moreover, since $\sN(X,\sE)$ is an $\sN(X)$-module (Corollary~\ref{cor:NearHolModule}), it therefore remains to show that $\sA:=\sN(X)$ is a dense subalgebra of $C(X)$. In other words, it suffices to prove Theorem~\ref{thm:NearlyHolDens} for the case of the trivial line bundle $\sE = X\times \CC$, noting that $\sN(X)$ is non-trivial since constant functions are nearly holomorphic.

We now give explicit generators for the subalgebra of nearly holomorphic functions. Let $\mf g=\Lie G$ be the Lie algebra of $G$, and let $\vartheta:\mf g\to\mf g$ denote the Cartan involution corresponding to the compact real form $\mf u = \Lie U$. Recall that $\mf g$ decomposes into $\mf g = \mf n^+\oplus\mf l\oplus\mf n^-$, where $\mf n^\pm$ are $\Ad_K$-invariant abelian subalgebras and $\mf l$ is the complexification of the Lie algebra $\mf k$ of $K$. Moreover, this decomposition is induced by a unique grading element $Z_0\in\mf z(\mf l)$ of the center of $\mf l$, so $\mf n^\pm = \Set{Y\in\mf g}{[Z_0,Y] = \pm Y}$. Let $\kappa:\mf g\times\mf g\to\CC$ be the Killing form of $\mf g$. The Levi-decomposition of the parabolic subgroup $P$ is given by $P = L\ltimes\exp(\mf n^-)$, where $L$ is reductive with Lie algebra $\mf l$. Consider the map
\begin{align}\label{eq:omegamap}
	\omega:\mf g\to C^\infty(X),\ Y\mapsto \omega(Y)(u):=\kappa(Y,\Ad_u Z_0)\;.
\end{align}
Since $\Ad_kZ_0=Z_0$ for all $k\in K$, it follows that $\omega(Y)$ indeed defines a smooth function on $X$.
Let $\Omega\subseteq C^\infty(X)$ be the unital subalgebra generated by $\omega(\mf g)$,
\begin{align}\label{eq:OmegaDef}
	\Omega := \CC[\omega(\mf g)]\subseteq C^\infty(X)\,.
\end{align}

Using the classical Stone--Weierstrass theorem, we show that $\Omega$ is dense in $C(X)$ (Proposition~\ref{prop:DensityOfOmega}), and by means of the Jordan theoretic description of Hermitian symmetric spaces we prove that $\Omega$ consists of nearly holomorphic functions (Proposition~\ref{prop:OmegaNearlyHolomorphic}). This completes the proof of Theorem~\ref{thm:NearlyHolDens}. 

\begin{proposition}\label{prop:DensityOfOmega}
	The subalgebra $\Omega$ is dense in $C(X)$ with respect to uniform convergence.
\end{proposition}
\begin{proof}
Due to the Stone--Weierstrass theorem, it suffices to show that $\Omega$ is conjugation invariant and separates points. Conjugation invariance follows from
\begin{align*}
	\overline{\omega(Y)(u)}
		&= \overline{\kappa(Y,\Ad_uZ_0)} \\
		&= \kappa(\vartheta(Y),\vartheta(\Ad_uZ_0)) \\
		&= \kappa(\vartheta(Y),\Ad_u(-Z_0))\\
		&=\omega(-\vartheta(Y))(u)\;.
\end{align*}
Concerning the separation of points, we first note that the $\Ad_G$-invari\-ance of the Killing form immediately implies that $\omega$ is $U$-equivariant. Therefore, it suffices to show separation of the base point $eK\in X$ from any other point $uK\neq eK$, i.e., we have to show that for all $u\in U\setminus K$, there exists an element $Y\in\mf g$ such that $\omega(Y)(u)\neq\omega(Y)(e)$, or equivalently $\kappa(Y,Z_0-\Ad_uZ_0)\neq 0$. Since $\Ad_uZ_0 = Z_0$ if and only if $u\in K$, this is a consequence of the non-degeneracy of the Killing form $\kappa$ on $\mf g$.
\end{proof}

Recall that the exponential map induces a holomorphic map,
\[
	\mf n^+\hookrightarrow X=G/P,\ z\mapsto\exp(z)P,
\]
onto an open and dense subset of $X$. When identifying $\mf n^+$ with its image, we simply write $\mf n^+\subseteq X$ in the sequel. Moreover, we identify the holomorphic tangent space $T^{(1,0)}_z$ at $z\in\Vp$ with $\Vp$. Since the Killing form restricts to a non-degenerate pairing $\kappa|:\mf n^+\times\mf n^-\to\CC$, we may also identify the holomorphic cotangent space $(T^{(1,0)}_z)^*\cong(\mf n^+)^*$ with $\Vm$ via the isomorphism
\begin{align}\label{eq:KillingIsomorphism}
	\mf n^-\to(\mf n^+)^*,\ w\mapsto \big(v\mapsto-\kappa(v,w)\big).
\end{align}
Here, the sign is introduced to avoid corresponding signs in the formulas for the Kähler potential and its $q$-map in Lemma~\ref{lem:HermMetricAndKaehlerPot}.

In order to determine the restriction of $\omega(Y)$ to $\mf n^+$ (and hence to show that $\omega(Y)$ is nearly holomorphic), we briefly recall some notions from Jordan theory. For a detailed introduction, we refer to \cite{Lo75,Lo77}, see also \cite{BN04}. In particular, we use the list of Jordan identities from the appendix of \cite{Lo77} and refer to single identities by \JP{xy}.

The Lie bracket on $\mf g$ defines a Jordan pair structure on $(\mf n^+,\mf n^-)$ given by
\[
	\mf n^\pm\times\mf n^\mp\times\mf n^\pm\to\mf n^\pm,\
	(x,y,z)\mapsto\JTP{x}{y}{z}_\pm:=-[[x,y],z]\;.
\]
As usual, we omit the indices $\pm$ on the triple product $\JTP{-}{-}{-}$, and define additional operators by
\[
	Q_xy:=\tfrac{1}{2}\JTP{x}{y}{x}
	\quad\text{and}\quad
	D_{x,y}z := Q_{x,z}y := \JTP{x}{y}{z}\;.
\]
The context determines the domains of these operators, e.g.\ for $x\in\mf n^+$ the operator $Q_x$ is a linear map from $\Vm$ to $\Vp$. In the following, let $\Tr_\pm T$ and $\Det_\pm T$ denote the trace and the determinant of endomorphisms on $\mf n^\pm$. In particular, if $h$ is an element of the Levi factor $L\subseteq P$, then the restrictions of the adjoint action of $h$ to $\mf n^+$ and $\mf n^-$ define isomorphisms,
\[
	h^\pm:\mf n^\pm\to\mf n^\pm,\ x\mapsto h^\pm x:=\Ad_hx\,.
\]
Likewise, if $T$ is an element of the Lie algebra $\mf l$ of $L$, then
\[
	T^\pm:\mf n^\pm\to\mf n^\pm,\ x\mapsto T^\pm x:=[T,x]
\]
are endomorphisms of $\mf n^+$ and $\mf n^-$. In fact, the pair $(h^+,h^-)$ (resp. $(T^+,T^-)$) is an \emph{automorphism} (resp. \emph{derivation}) of the Jordan pair $(\mf n^+,\mf n^-)$. As an example, we note that for $x\in\mf n^+$ and $y\in\mf n^-$ we have $T:=[x,y]\in\mf l$ with $T^+ = D_{x,y}$ and $T^- = -D_{y,x}$.
By abuse of notation, we omit the indices $\pm$ and simply write $hx=\Ad_hx$ and $Tx=[T,x]$ for $x\in\mf n^\pm$. Using the Killing form it is straightforward to show that $\Det_- h = (\Det_+ h)^{-1}$ and $\Tr_- T = -\Tr_+ T$.
 
Via the open and dense embedding $\mf n^+\subseteq X$, we may identify the elements of the Lie algebra with certain holomorphic vector fields on $\mf n^+$, i.e., holomorphic maps $\zeta\in\sO(\mf n^+,\mf n^+)$. More precisely,
\begin{align*}
	\mf n^+ &\cong \Set{\zeta(z) = u}{u\in\Vp} & &\text{(constant vector fields),}\\
	\mf n^- &\cong \Set{\zeta(z) = Q_zv}{v\in\Vm}& &\text{(quadratic vector fields),}
\end{align*}
and $\mf l$ corresponds to linear vector fields, $\zeta(z) = Tz$ for $T\in\mf l$. According to the decomposition $\mf g = \mf n^+\oplus\mf l\oplus\mf n^-$, we write $(u,T,v)$ with $u\in\Vp$, $T\in\mf l$, $v\in\Vm$ for elements in $\mf g$. The distinguished element $Z_0\in\mf l$ is identified with the vector field corresponding to the identity, so $Z_0=(0,\Id_{\Vp},0)$. The commutator of two elements $X_1= (v_1,T_1,w_1)$ and $X_2 = (v_2,T_2,w_2)$ is given by
\begin{align*}
	[X_1,X_2] = \big(T_1v_2-T_2v_1,\ D_{v_2,w_1} + [T_1,T_2] - D_{v_1,w_2},\ 
							T_2w_1-T_1w_2\big),
\end{align*}
and the Killing form $\kappa$ on $\mf g$ translates to 
\begin{align}\label{eq:KillingForm}
	\kappa(X_1,X_2) = \kappa_\mf l(T_1,T_2) + 2\,\Tr_+(T_1T_2 - D_{v_1,w_2} - D_{v_2,w_1})\;,
\end{align}
where $\kappa_\mf l$ denotes the Killing form on $\mf l$, see e.g.\ \cite[I.7.8]{Sat80}.

 A particularly important operator is the \emph{Bergman operator} defined by
\begin{align}\label{eq:BergmanOperator}
	\B{x}{y}:\mf n^+\to\mf n^+,\ \B{x}{y}:=\Id_{\mf n^+} - D_{x,y} + Q_xQ_y
\end{align}
for $x\in\mf n^+$, $y\in\mf n^-$. Its determinant is the power of an irreducible polynomial, $\Delta$, on $\mf n^+\times\mf n^-$,
\begin{align}\label{eq:BergmanAndDeterminant}
	\Det_+\B{x}{y} = \Delta(x,y)^p.
\end{align}
The polynomial $\Delta$ is called the \emph{Jordan pair determinant} (or the \emph{generic norm}, cf.\ \cite{Lo75}), and the exponent $p$ is one of the structure constants of the irreducible Hermitian symmetric space. The pair $(x,y)\in\mf n^+\times\mf n^-$ is called \emph{quasi-invertible} if $\Delta(x,y)\neq 0$, or equivalently if $\B{x}{y}$ is invertible. In this case,
\begin{align}\label{eq:quasiinverse}
	x^y:=\B{x}{y}^{-1}(x-Q_xy)\in\mf n^+
\end{align}
is the \emph{quasi-inverse} of $(x,y)$. In fact, the quasi-inverse is precisely the element obtained by the action of $\exp(y)\in G$ on $x\in\mf n^+\subseteq X$, see \cite[\S\,8.4]{Lo77}. Exchanging $x$ and $y$ in \eqref{eq:BergmanOperator} and \eqref{eq:quasiinverse}, defines the (dual) Bergman operator $\B{y}{x}\in\End(\mf n^-)$ and the corresponding quasi-inverse $y^x\in\mf n^-$, if it exists. It is well-known that $\B{x}{y}$ is invertible if and only if $\B{y}{x}$ is invertible, and in this case there is an element of the Levi factor $h\in L$ such that $(h^+,h^-) = (\B{x}{y},\B{y}{x}^{-1})$. By abuse of notation, we may write $\B{x}{y}\in L$. 

The restriction of the Cartan involution $\vartheta$ of $\mf g$ to $\mf n^+$ and $\mf n^-$ yields isomorphisms $\vartheta|:\mf n^\pm\to\mf n^\mp$, which are mutual inverses. For convenience, both isomorphisms are denoted by $x\mapsto\bar x$ for $x\in\mf n^\pm$. It is well-known that for all $z\in\mf n^+$ the Bergman operator $\B{z}{-\bar z}$ is a positive definite operator with respect to the Hermitian inner product
\begin{align}\label{eq:InnerProduct}
	(v|w):=-\kappa(v,\bar w)\quad\text{with}\quad v,w\in\mf n^+,
\end{align}
see e.g.\ \cite[\S\,3.15]{Lo77}. In the realization of $\mf g$ as vector fields on $\Vp$, the Cartan involution on $\mf g$ reads
\[
	\vartheta(u,T,v) = (\overline v,-T^*,\overline u)\,,
\]
where $T^*\in\End(\Vp)$ denotes the adjoint of $T\in\End(\Vp)$ with respect to the inner product \eqref{eq:InnerProduct}.

\begin{lemma}\label{lem:ChartLift}
	The map
	\[
		\mf n^+\to U,\ z\mapsto u_z:=\exp(z)\cdot\B{z}{-\bar z}^\half\cdot\exp(\bar z)
	\]
	defines a smooth lift of the embedding $\mf n^+\hookrightarrow X$, i.e., the following diagram 
	is commutative
	\begin{center}
	\begin{tikzpicture}[description/.style={fill=white,inner sep=2pt},
											baseline=0pt]
		\matrix (m) [matrix of math nodes, row sep=3em,
								 column sep=2.5em, text height=1.5ex, text depth=0.25ex, outer sep=4pt]
		{  & U \\
			\mf n^+ & G=X/P \\};
		\path[right hook->] (m-2-1) edge (m-1-2);
		\path[->>] (m-1-2) edge (m-2-2);
		\path[right hook->] (m-2-1) edge (m-2-2);
	\end{tikzpicture}
	\quad with\qquad
	\begin{tikzpicture}[description/.style={fill=white,inner sep=2pt},
											baseline=0pt]
		\matrix (m) [matrix of math nodes, row sep=3em,
								 column sep=2.5em, text height=1.5ex, text depth=0.25ex, outer sep=8pt]
		{  & u_z \\
			z & \exp(z)P = u_z P\ . \\};
		\path[|->] (m-2-1) edge (m-1-2);
		\path[|->] (m-1-2) edge (m-2-2);
		\path[|->] (m-2-1) edge (m-2-2);
	\end{tikzpicture}
	\end{center}
\end{lemma}
\begin{proof}
It follows from \cite[\S\,9.7f]{Lo77} that $u_z$ can be written as $u_z = \exp(w + \bar w)$ for some $w\in\mf n^+$, and that $\B{z}{-\bar z}^\half$ is an element of the Levi factor $L\subseteq P$. Therefore, $u_z$ is indeed an element of $U$, and since $P=L\ltimes\exp(\mf n^-)$, we readily obtain the commutativity of the diagram.
\end{proof}

The following lemma provides explicit formulas for the Kähler metric of $X$ (more precisely, for the corresponding sesquilinear form on the holomorphic tangent space) and for a Kähler potential on the open and dense subset $\mf n^+\subseteq X$. Up to a normalizing constant, the $U$-invariant Kähler metric is unique.

\begin{lemma}\label{lem:HermMetricAndKaehlerPot}
	For all $z\in\mf n^+\subseteq X$, the Kähler metric is given by
	\[
		h_z:\mf n^+\times\mf n^+\to\CC,\ h_z(v,w)=(\B{z}{-\bar z}^{-1}v|w).
	\]
	Moreover, $\Psi:\mf n^+\to\RR$ defined by
	\[
		\Psi(z):=2p\,\log \Delta(z,-\bar z)
	\]
	is a Kähler potential with corresponding $q$-map, $q:\mf n^+\to\mf n^-$, given by
	\[
		q(z)=\partial\Psi(z) = \bar z^{-z}\,,
	\]
	where we have identified $(\mf n^+)^*$ with $\mf n^-$ via \eqref{eq:KillingIsomorphism}.
\end{lemma}
\begin{proof}
At the base point $o=eP\in X$, the Kähler metric is $K$-invariant, and since the inner product $(-|-)$ defined in \eqref{eq:InnerProduct} is the only $K$-invariant inner product on $\mf n^+$ (up to a constant), it follows that $h_0(v,w) = (v|w)$. For $z\in\Vp$, Lemma~\ref{lem:ChartLift} yields that 
\[
	h_z(v,w) = h_0(\partial u_z(0)^{-1}v,\partial u_z(0)^{-1}w)
	= (\B{z}{-\bar z}^{-\half}v,\B{z}{-\bar z}^{-\half}w) = (\B{z}{-\bar z}^{-1}v,w)\,,
\]
where we have used that $u_z(x) = z+\B{z}{-\bar z}^\half x^{\bar z}$, and the holomorphic derivative of $(x\mapsto x^{\bar z})$ is given by $\partial_v(x^{\bar z}) = \B{x}{\bar z}^{-1}v$, see \cite[\S\,7.8]{Lo77}. Next, we determine $\del_v\Psi$. Since $\tfrac{d}{dt}(\Det_+ A_t) = \Det_+ A_t\cdot\Tr_+(A_t^{-1}\tfrac{d}{dt} A_t)$ for any smooth curve $t\mapsto A_t$ into $\End(\mf n^+)$, we obtain
\begin{align*}
	\del_v\Psi &= 2\,\del_v\log\Det_+\B{z}{-\bar z} 
		= 2\,\Tr_+\big(\B{z}{-\bar z}^{-1}(-D_{v,\bar z} + Q_{v,z}Q_{\bar z})\big)\;.
\end{align*}
Due to \JP{31}, this simplifies to $\del_v\Psi=2\Tr_+ D_{v,{\bar z}^{-z}}$, and due to \eqref{eq:KillingForm}, we conclude that $\del_v\Psi = -\kappa(v,{\bar z}^{-z})$ and hence $q(z) = \bar z^{-z}$. Finally,
\[
	\delbar_{\bar w}\del_v\Psi(z)
		= -\delbar_{\bar w}\kappa(v,{\bar z}^{-z})
		= -\kappa(v,\B{\bar z}{-z}^{-1}\bar w)
		= -\kappa(\B{z}{-\bar z}^{-1}v,\bar w)
		= (\B{z}{-\bar z}^{-1} v, w)\,.
\]
This completes the proof.
\end{proof}

Now, we are prepared to determine the local description of the smooth functions on $X$ defined by $\omega$ in \eqref{eq:omegamap}. 

\begin{proposition}\label{prop:OmegaNearlyHolomorphic}
	For any $Y\in\mf g$, the function $\omega(Y)\in C^\infty(X)$ is nearly holomorphic.
	More precisely, for $Y=(v,T,w)\in\mf g$, the restriction of $\omega(Y)$ to $\Vp\subseteq X$ is 
	given by
	\begin{align*}
		\omega(Y)\at{\Vp}(z) 
			&= 2\,\Tr_+ T -\kappa(z,w) +\kappa(v-Tz + Q_zw,\,q(z)).
	\end{align*}
	Therefore, $\Omega\subseteq\sN(X)$.
\end{proposition}
\begin{proof}
By definition, $\omega(Y)$ is a smooth function on $X$. Since $\Vp\subseteq X$ is open and dense, it follows that $\omega(Y)$ is nearly holomorphic on $X$ if and only if the restriction $\omega(Y)|_\Vp$ is nearly holomorphic on $\Vp$. It therefore suffices to proof the local formula for $\omega(Y)$. Due to Lemma~\ref{lem:ChartLift}, the restriction of $\omega(Y)$ to $\Vp$ is given by
\[
	\omega(Y)\at{\Vp}(z) = \kappa(Y,\Ad_{u_z}Z_0)
\]
with $u_z:=\exp(z)\cdot\B{z}{-\bar z}^\half\cdot\exp(\bar z)$. Realizing elements of $\mf g$ as holomorphic vector fields on $\Vp$, the adjoint action reads
\[
	(\Ad_{u^{-1}}\zeta)(x) = \partial u(x)^{-1}\cdot\zeta(u(x))
\]
for $u\in U$, $\zeta\in\mf g$ and $x\in\Vp$. Since $u_z^{-1}(x) = u_{-z}(x) = \B{z}{-\overline z}^\half x^{-\overline z} - z$, this yields
\begin{align*}
	(\Ad_{u_z}Z_0)(x)
		&= \big(\B{z}{-\overline z}^\half
							 \B{x}{-\overline z}^{-1}\big)^{-1}\cdot\big(u_z^{-1}(z)\big)\\
		&= \B{x}{-\bar z} x^{-\bar z} - \B{x}{-\bar z}\B{z}{-\bar z}^{-\half} z \\
		&= x+Q_x\bar z - \B{x}{-\bar z}(z^{-\bar z}) \\
		&= -z^{-\bar z} + x + \JTP{z^{-\bar z}}{\bar z}{x} + Q_x(\bar z - Q_{\bar z} z^{-\bar z})\\
		&= -z^{-\bar z} + (\Id_\Vp + D_{z^{-\bar z},\bar z})(x) + Q_x{\bar z}^{-z}\;.
\end{align*}
Here, the identity $\B{z}{-\bar z}^{-\half}z = z^{-\bar z}$ follows from $\exp(-\bar z)u_z = \B{z}{-\bar z}^{-\half}\exp(z)$ which is a consequence of $u_z = u_{-z}^{-1}$.
We thus obtain that 
\[
	\Ad_{u_z}Z_0 = Z_0 + (-z^{-\bar z},\,D_{z^{-\bar z},\bar z},\,{\bar z}^{-z})\in i\mf u.
\]
For the evaluation of the Killing form, we use the identity 
\begin{align*}
	(-z^{-\bar z},D_{z^{-\bar z},\bar z},{\bar z}^{-z})
		= [(z^{-\bar z},0,0),(0,\Id_\Vp,-\tfrac{1}{2}\bar z)]
		+ [(0,0,{\bar z}^{-z}),(\tfrac{1}{2}z,\Id_\Vp,0)]
\end{align*}
and the associativity of the Killing form. We conclude that
\begin{align*}
	\kappa(Y,\Ad_{u_z}Z_0) &= 2\Tr T -\kappa(z^{-\bar z},w) +\kappa(v,\bar z^{- z})\\
	 &\quad-\tfrac{1}{2}\kappa(Tz^{-\bar z},\bar z) -\tfrac{1}{2}\kappa(Tz,\bar z^{-z}).
\end{align*}
Due to the identity $z^{-\bar z} = z - Q_z\bar z^{-z}$ (see the symmetry formula in \cite[Appendix]{Lo77}), the term $\kappa(z^{-\bar z},w)$ can be rewritten as
\[
	\kappa(z^{-\bar z},w) = \kappa(z,w) - \kappa(Q_z\bar z^{-z},w)
	= \kappa(z,w) - \kappa(Q_zw,\bar z^{-z})\;.
\]
To show that the last two terms of $\kappa(Y,\Ad_{u_z}Z_0)$ coincide, it suffices to assume $T = D_{x,y}$ since $\mf l = [\mf n^+,\mf n^-]$. Using the relation $D_{\bar z,z^{-\bar z}} = D_{\bar z^{-z},z}$ which follows from $z^{-\bar z} = z - Q_z\bar z^{-z}$ and \JP{32}, and also using the associativity of the Killing form (recall that $D_{x,y}z = -[[x,y],z]$), we obtain
\begin{align*}
	\kappa(D_{x,y}z^{-\bar z},\bar z)
		= \kappa(x,D_{\bar z,z^{-\bar z}}y)
		= \kappa(x,D_{\bar z^{-z},z}y)
		= \kappa(D_{x,y}z,\bar z^{-z})\;.
\end{align*}
Putting all terms together this yields the stated formula.
We thus have proved that $\omega(Y)\at{\Vp}$ is indeed nearly holomorphic (of degree $1$) on $\Vp$. 
\end{proof}

In the next section, we use a simple representation theoretic argument to show that $\Omega$ actually coincides with the algebra of nearly holomorphic functions. Then, Proposition~\ref{prop:OmegaNearlyHolomorphic} yields the following description of nearly holomorphic functions, which was first obtained by Shimura for the classical Hermitian symmetric spaces of compact type by a case by case analysis, see Theorem~2.3 in \cite{S87}.

\begin{corollary}\label{cor:nearlyholomorphicgenerators}
	Let $(c_1,\ldots,c_n)$ be a basis of $\Vp$, and $(\tilde c_1,\ldots,\tilde c_n)$ be the 
	corresponding dual basis of $\Vm$ with respect to the Killing form $\kappa$. Set 
	\[
		q_\ell(z):=\kappa(c_\ell, q(z)),\quad
		d_{\ell k}(z):=\kappa(D_{z,\,q(z)}c_\ell,\tilde c_k)\,.
	\]
	Then, the space of nearly holomorphic functions satisfies
	\[
		\sN(X)
			= \CC\big[\Set{q_\ell,\ \overline{q_k},\ d_{\ell k}}{\ell,k=1,\ldots,n}\big]\,.
	\]
	In other words, the local picture of each nearly holomorphic function is a polynomial in the 
	terms $q_\ell$, $\overline{q_k}$, $d_{\ell k}$, and conversely, each polynomial in these terms 
	represents a nearly holomorphic function.
\end{corollary}
\begin{proof}
We first note that constant functions are nearly holomorphic. Then, Remark~\ref{rmk:OmegaEquality}, Proposition~\ref{prop:OmegaNearlyHolomorphic} and its proof show that the space of nearly holomorphic functions is algebraically generated by the following elements
\begin{align*}
	\omega((v,0,0))\at{\Vp}(z) &= \kappa(v,q(z)),\\
	\omega((0,0,w))\at{\Vp}(z) &= -\kappa(\overline{q(z)},w),\\
	\omega((0,T,0))\at{\Vp}(z) &= 2\,\Tr_+ T  -\kappa(Tz,q(z))\,,
\end{align*}
with $v\in\Vp$, $w\in\Vm$, $T\in\mf l$. Since $\kappa(x,\bar y) = \overline{\kappa(y,\bar x)}$, the first two sets of generators correspond to the coefficient functions $q_\ell(z)$ and $\overline{q_k(z)}$. For the last type of generator, it suffices to consider $T=D_{v,w}$ for $v\in\Vp$, $w\in\Vm$, since $\mf l$ is (linearly) generated by such elements. The constant term $2\,\Tr_+ T$ can be ignored, since (by definition) constants take part of the generators of a unital $\CC$-algebra. Finally, the associativity of the Killing form implies 
\[
	\kappa(D_{v,w} z,q(z)) = \kappa(D_{v,q(z)}z,w) = \kappa(D_{z,q(z)}v,w),
\]
so the generators corresponding to elements in $\mf l$ can be replaced by the set of all $d_{\ell k}(z)$, $k,\ell = 1,\ldots, n$.
\end{proof}


\section{Application to harmonic analysis}\label{sec:RepTheoApp}
In this section, we apply the results of the last sections to representation theoretic questions on the Hilbert space $L^2(X,\sE)$ of square integrable sections in a $G$-homogeneous irreducible Hermitian vector bundle $\sE=G\times_P E$. 

\subsection{Preliminaries}
As before, let $X=G/P=U/K$ be an irreducible Hermitian symmetric space. Let $\mf h\subseteq\mf g$ be a Cartan subalgebra of $\mf g$ and $\Phi^+(\mf g,\mf h)\subseteq\Phi(\mf g,\mf h)$ be a system of positivity in the corresponding root system , such that $P$ is the standard parabolic subgroup corresponding to the simple root $\alpha_1\in\Phi^+(\mf g,\mf h)$. More precisely, if $\Delta = \{\alpha_1,\ldots,\alpha_\ell\}$ is the set of simple roots, then any $\beta\in\Phi(\mf g,\mf h)$ has a unique decomposition $\beta = \sum m_i(\beta)\alpha_i$, and
\[
	\mf p = \bigoplus_{\beta\in\Phi(\mf g,\mf h):\ m_1(\beta)\leq 0} \mf g_\beta\,,
\]
where $\mf g_\beta$ denotes the root space of $\beta$. In this context, the assumption that $X$ is Hermitian symmetric is equivalent to the condition that $|m_1(\beta)|\leq 1$ for all $\beta$, which implies that the components of the decomposition $\mf g = \mf n^+\oplus\mf l\oplus\mf n^-$ (as above) are given by 
\[
	\mf l = \mf h\oplus\bigoplus_{\beta\in\Phi(\mf g,\mf h):\ m_1(\beta)=0}\mf g_\beta\,,\quad
	\mf n^\pm = \bigoplus_{\beta\in\Phi(\mf g,\mf h):\ m_1(\beta)=\pm1}\mf g_\beta\,.
\]
We note that $\mf h$ is also a Cartan subalgebra of $\mf l$ with root system $\Phi(\mf l,\mf h)\subseteq\Phi(\mf g,\mf h)$, consisting of the \emph{compact} roots. We set
\[
	\Phi:=\Phi(\mf g,\mf h),\quad\Phi_c:=\Phi(\mf l,\mf h),\quad\Phi_{nc}:=\Phi\setminus\Phi_c,
\]
and write $\Phi^+$, $\Phi^+_c$, $\Phi^+_{nc}$ for the corresponding subsets of positive roots. The roots in $\Phi_{nc}$ are called \emph{non-compact}. The positive (resp. negative) non-compact roots are those with root spaces lying in $\mf n^+$ (resp. $\mf n^-$). For each root $\alpha\in\Phi^+$, we fix a corresponding $\mf{sl}_2$-triple $(X_\alpha,H_\alpha,Y_\alpha)$,
\[
	[H_\alpha,X_\alpha] = 2\,X_\alpha,\quad
	[H_\alpha,Y_\alpha] = -2\,Y_\alpha,\quad
	[X_\alpha,Y_\alpha] = H_\alpha\,.
\]
For simple roots, we set $(X_i,H_i,Y_i):=(X_{\alpha_i},H_{\alpha_i},Y_{\alpha_i})$. The Cartan subalgebra $\mf h\subseteq\mf g$ splits into $\mf h = \mf z(\mf l)\oplus\mf h'$, where $\mf z(\mf l)$ is the center of $\mf l$, and $\mf h':=\CC H_2\oplus\cdots\oplus\CC H_\ell$ is a Cartan subalgebra of the semisimple part $\mf l_\ss$ of $\mf l$, i.e., $\mf l = \mf z(\mf l)\oplus\mf l_\ss$. The center $\mf z(\mf l)$ is spanned by the distinguished element $Z_0\in\mf l$ and satisfies $\mf z(\mf l) = \set{H\in\mf h}{\alpha_i(H) = 0\ \forall i>1}$. 

The set of dominant integral weights with respect to $\Phi^+$ is denoted by
\[
	\Lambda := \Set{\lambda\in\mf h^*}{\lambda(H_i)\in\NN\text{ for all $i=1,\ldots,\ell$}}\,.
\]
Since $G$ is simply-connected, $\Lambda$ parametrizes the set of all isomorphism classes of irreducible finite dimensional representations of $G$. A representative of the isomorphism class corresponding to $\lambda\in\Lambda$ is denoted by $V_\lambda$. Let $\Lambda_c\subseteq\mf h^*$ denote the set of highest weights corresponding to finite dimensional irreducible representations of $L$. This is a subset of the set of dominant analytically integral weights with respect to $\Phi^+_c$. The isomorphism class corresponding to $\lambda\in\Lambda_c$ is denoted by $E_\lambda$. We note that $\Lambda\subseteq\Lambda_c$, and for $\lambda\in\Lambda$ we have
\[
	V_\lambda^{\mf n^+} :=\Set{v\in V_\lambda}{Y.v = 0\text{ for all }Y\in\mf n^+}
	\cong E_\lambda
\]
as $L$-modules, see \cite[V.\S7]{Kna02}.

Recall that continuous (resp. smooth or holomorphic) sections in $\sE=G\times_P E$ can be described as continuous (resp. smooth or holomorphic) functions $f:G\to E$ satisfying the equivariance condition
\begin{align}\label{eq:equivariance}
	f(gp) = \rho(p)^{-1}f(x)\quad\text{for all}\quad g\in G,\ p\in P.
\end{align}
Recall that Hermitian structures on $\sE$ correspond to $K$-invariant Hermitian inner products on $E$. Let $|\cdot|$ denote the corresponding $K$-invariant norm on $E$. In this description, the $L^2$-norm is given by
\[
	\|f\|^2 = \int_U |f(u)|^2du\,,
\]
where $du$ is the suitably normalized invariant measure on $U$. Then, $L^2(X,\sE)$ is the completion of $C(X,\sE)$ with respect to this norm. The standard action of $G$ on continuous sections reads $(g.f)(g') = f(g^{-1}g')$. This actions extends to a continuous representation of $G$ on $L^2(X,\sE)$, denoted by $\pi$, which restricts to a unitary representation of $U\subseteq G$. 

If $f:G\to E$ represents a continuous section in $\sE$, then its restriction to the open and dense subset $\mf n^+\subseteq X$ is given by
\begin{align}\label{eq:LocalVsInducedPicture}
	f_{\mf n^+}(z)=f(\exp(z))\quad\text{for}\quad z\in\mf n^+,
\end{align}
which is a continuous map $f_{\mf n^+}:\mf n^+\to E$. The $L^2$-norm translates to an integral on $\mf n^+$, whose explicit form is not needed in the sequel. However, let $L^2(\mf n^+,E)$ denote the corresponding $L^2$-space, which is isomorphic to $L^2(X,\sE)$ (with appropriate $G$-action), and which is called the \emph{local picture} of $L^2(X,\sE)$.

\begin{lemma}
	In the local picture, the $G$-action on $f\in L^2(\Vp,E)$ reads
	\begin{align*}
		\pi(\exp(v))f(z) &= f(z-v)\,,\\
		\pi(h)f(z) &= \rho(h)\,f(h^{-1}z)\,,\\
		\pi(\exp(w))f(z) &= \rho(\B{z}{-w})\,f(z^{-w})
	\end{align*}
	with $v\in\Vp$, $h\in L$, $w\in\Vm$.
\end{lemma}

\begin{proof}
This easily follows from the relation \eqref{eq:LocalVsInducedPicture}, the equivariance property \eqref{eq:equivariance}, and the identities 
\begin{align*}
	h^{-1}\exp(z) &= \exp(h^{-1}z)h^{-1},\\
	\exp(-w)\exp(z) &= \exp(z^{-w})\B{z}{-w}^{-1}\exp((-w)^z)
\end{align*}
for $h\in L$, $z\in\Vp$ and $w\in\Vm$, see \cite[\S\,8.11]{Lo77}.
\end{proof}
The representation of the Lie algebra $\mf g$ induced by the $G$-action on smooth sections is readily obtained by differentiation,
\begin{align*}
	d\pi(u,0,0)f(z) &= -d_uf(z)\,,\\
	d\pi(0,T,0)f(z) &= -d_{Tz}f(z) + d\rho(T)\,f(z)\,,\\
	d\pi(0,0,v)f(z) &= -d_{Q_zv}f(z) + d\rho(D_{z,v})\,f(z)\,.
\end{align*}
It is important to note that $d_u$ denotes the \emph{real} directional derivative along $u$. In order to investigate complex structures, it is convenient to consider the embedding $\Vp\hookrightarrow\Vp\times\Vm$ given by $z\mapsto(z,\bar z)$. Then, the complexified tangent space at $(z,\bar z)$ is identified with $\Vp\times\Vm$ where $\Vp$ is the holomorphic and $\Vm$ is the anti-holomorphic component. In this way, $d_u$ (or more appropriately $d_{(u,\bar u)}$) becomes $\del_u + \delbar_{\bar u}$. In particular, the element $Y=(v,T,\bar v)\in\mf u$ with $v\in\mf n^+$, $T\in\mf k$, acts by
\[
	d\pi(Y)f = -\partial_{v+Tz+Q_z\bar v}f 
						 -\bar\partial_{\bar v+ T\bar z + Q_{\bar z}v}f
						 +d\rho(T + D_{z,\bar v})\,f(z)\;.
\]
For the discussion of weights corresponding to this $\mf u$-action, we are interested in its complexification. Since $\mf u_\CC =\mf g$, it is likely to get confused about the complexified action of $\mf u_\CC$ and the action of $\mf g$ discussed above. We denote the complexified action by $d\pi_\CC$. The full complexified $\mf g$-action is not needed in the sequel.

\begin{lemma}\label{lem:localLiealgebraAction}
	In the local picture, the complexified action of $\mf u^\CC=\mf g$ is given by
	\begin{align*}
		d\pi_\CC(v,0,0)f(z) &= \big(-\partial_v-\bar\partial_{Q_{\bar z} v}\big)f(z)\,,\\	
		d\pi_\CC(0,T,0)f(z) &= \big(-\partial_{Tz}-\bar\partial_{T\bar z} + d\rho(T)\big)f(z)\,,\\
		d\pi_\CC(0,0,w)f(z) &= \big(-\partial_{Q_zw}-\bar\partial_w + d\rho(D_{z,w})\big)f(z)
	\end{align*}
	with $v\in\Vp$, $T\in\mf l$, $w\in\Vm$.
\end{lemma}
Using the identification of $\Vp$ with the diagonal in $\Vp\times\Vm$, the proof of this proposition is straightforward, we therefore omit the details.

\subsection{Main result}
Standard representation theory for compact Lie groups yields that $L^2(X,\sE)$ decomposes into the Hilbert sum
\begin{align}\label{eq:L2decomposition}
	L^2(X,\sE)= \widehat{\bigoplus_{\lambda\in\Lambda}}\; W_\lambda^\sE\,
\end{align}
of $U$-isotypic components $W_\lambda^\sE\cong m_\lambda^\sE\cdot V_\lambda$, where $m_\lambda^\sE\geq 0$ denotes the multiplicity of $V_\lambda$ in $L^2(X,\sE)$. According to Frobenius reciprocity, $m_\lambda^\sE$ is finite, so $W_\lambda^\sE$ are finite dimensional. Let $W^\sE$ denote the algebraic sum of the isotypic components,
\[
	W^\sE:=\bigoplus_{\lambda\in\Lambda} W_\lambda^\sE\;,
\]
which forms a $U$-invariant dense subspace of $L^2(X,\sE)$, and can also be characterized as the set of $U$-finite vectors in $L^2(X,\sE)$.

\begin{proposition}\label{prop:NearlyHolomorphicAndUFinite}
	If $\sN(X,\sE)\neq\{0\}$, then $\sN(X,\sE)$ coincides with the space $W^\sE$ of $U$-finite 
	vectors in $L^2(X,\sE)$.
\end{proposition}
\begin{proof}
Since $\sN^m(X,\sE)$ is $U$-invariant and finite dimensional (Propositions~\ref{prop:BasisPropsNearHol} and \ref{prop:FiniteDim}), it follows that $\sN(X,\sE)\subseteq W^\sE$. More precisely, if $\sN(X,\sE) = \bigoplus_{\lambda\in\Lambda}\sN_\lambda^\sE$ denotes the decomposition of $\sN(X,\sE)$ into $U$-isotypic components, then $\sN^\sE_\lambda\subseteq W^\sE_\lambda$ for all $\lambda$. Since $\sN(X,\sE)$ is assumed to be non-trivial, Theorem~\ref{thm:NearlyHolDens} implies that $\sN(X,\sE)$ is dense in $L^2(X,\sE)$. Since $W_\lambda^\sE$ is finite dimensional, this yields that for each $\lambda\in\Lambda$, the orthogonal projection of $\sN(X,\sE)$ to $W^\sE_\lambda$ is onto. Thus, $\sN^\sE_\lambda = W^\sE_\lambda$, which completes the proof.
\end{proof}

\begin{remark}\label{rmk:OmegaEquality}
	For the trivial line bundle, $\sE = X\times\CC$, the same argument as used for this proposition 
	also shows that the subspace $\Omega\subseteq L^2(X)$ defined in \eqref{eq:OmegaDef} coincides 
	with the space of $U$-finite vectors in $L^2(X)$. Therefore, $\sN(X) = \Omega = W^\sE$.
\end{remark}

Due to Proposition~\ref{prop:LocalDescription}, a nearly holomorphic section $f\in\sN^m(X,\sE)$ restricts on $\Vp\subseteq X$ to a map of the form 
\begin{align}\label{eq:localnearlyhol}
	f(z) = \sum_{|\b i|\leq m} f_\b i(z)\cdot q(z)^\b i\;,
\end{align}
with unique holomorphic coefficients $f_\b i\in\sO(\Vp,E)$, and $q:\mf n^+\to\mf n^-$ given by $q(z) = \bar z^{-z}$, see Lemma~\ref{lem:HermMetricAndKaehlerPot}. Moreover, according to the identity theorem this provides an embedding of $\sN(X,\sE)$ into $\sO(\Vp,E)[\Vm]$, the space of polynomials on $\Vm$ with coefficients in $\sO(\Vp,E)$,
\begin{align}\label{eq:nearlyholembedding}
	\iota_{\sN}: \sN(X,\sE)\hookrightarrow\sO(\Vp,E)[\Vm],\
	f\mapsto p_f(y):=\sum_{|\b i|\leq\deg f}f_\b i \,y^\b i\,.
\end{align}
By means of $\sO(\Vp,E)[\Vm]\subseteq\sO(\Vp\times\Vm,E)$, we may also write
\[
	p_f(x,y):=\sum_{|\b i|\leq\deg f}f_\b i(x) \,y^\b i\,.
\]
Let $\sP(\mf n^-,E)$ denote the space of $E$-valued complex polynomials on $\mf n^-$, which we may regard as a subspace of $\sO(\Vp,E)[\Vm]$. Consider the action of the Levi factor $L$ on $p\in \sO(\Vp\times\Vm,E)$ given by
\[
	(h.p)(x,y):=\rho(h)p(h^{-1}x,h^{-1}y)\quad\text{for}\quad h\in L,\ x\in\mf n^+,\ y\in\mf n^-\,,
\]
where $h^{-1}x = Ad_{h^{-1}}x$ and $h^{-1}y = Ad_{h^{-1}}y$ as in Section~\ref{sec:NearHolOnHermSym}. Then, $\sP(\mf n^-,E)$ and $\sO(\Vp,E)[\Vm]$ are $L$-invariant subspaces. We note that due to \cite[\S\,7.3]{Lo77}, 
\[
	q(kz) = \overline{kz}^{-kz} = k\,\bar z^{-z} = kq(z)
\]
for all $k\in K\subseteq L$. Therefore the embedding $\iota_\sN$ is $K$-equivariant.

\begin{theorem}\label{thm:highestweights}
	Assume that $\sN(X,\sE)\neq\{0\}$. For all $\lambda\in\Lambda$, the map
	\begin{align*}
		\varphi_\lambda:\Hom_U(V_\lambda,L^2(X,\sE))\to\Hom_K(V_\lambda^{\mf n^+},\Poly(\Vm,E)),\
		T\mapsto \iota_{\sN}\circ T|_{V_\lambda^{\mf n^+}}
	\end{align*}
	is well-defined and a vector space monomorphism. If, in addition, $\sO(X,\sE)\neq\{0\}$, then 
	$\varphi_\lambda$ is an isomorphism for all $\lambda\in\Lambda$, and all highest weights with 
	respect to $\Phi^+_c$ occurring in $\sP(\mf n^-,E)$ are also dominant integral for $\Phi^+$.
\end{theorem}
\begin{proof}
The prove of this theorem relies on the crucial fact that
\begin{align}\label{eq:qvanishing}
	(\partial_v+\bar\partial_{Q_{\bar z} v})q(z) = 0\quad\text{for all $v\in\mf n^+$.}
\end{align}
Indeed, due to \cite[\S\,7.8]{Lo77}, the holomorphic and anti-holomorphic derivatives of $q(z) = \bar z^{-z}$ are given by 
\[
	\del_vq(z) = -\B{\bar z}{-z}^{-1}Q_{\bar z} v,\quad
	\delbar_wq(z) = \B{\bar z}{-z}^{-1}w,
\]
which implies \eqref{eq:qvanishing}. Now, for $T\in\Hom_U(V_\lambda,L^2(X,\sE))$ and $f\in\Im T$ it follows from Proposition~\ref{prop:NearlyHolomorphicAndUFinite} that $f$ is nearly holomorphic. Then, \eqref{eq:qvanishing} and Lemma~\ref{lem:localLiealgebraAction} yields
\[
	d\pi_\CC(v,0,0)f(z) = \sum_{|i|\leq m} \partial_vf_\b i(z)\cdot q(z)^\b i\,.
\]
Since the coefficient sections of nearly holomorphic sections are unique (Proposition~\ref{prop:LocalDescription}), we conclude that $d\pi_\CC(v,0,0)f = 0$ for all $v\in\mf n^+$ if and only if $f_\b i$ are constant for all $\b i$, which is equivalent to the condition that $p_f=\iota_\sN(f)$ is an element of $\sP(\mf n^-,E)$. This shows that $\iota_\sN\circ T$ maps $f\in V_\lambda^{\mf n^+}$ into $\Poly(\mf n^-,E)$, and since $\iota_\sN$ is $K$-equivariant, we conclude that $\varphi_\lambda$ is well-defined. Obviously, $\varphi_\lambda$ is linear. In order to prove injectivity, assume that $\varphi_\lambda(T) = 0$. Since $\iota_\sN$ is injective, it follows that $T|_{V_\lambda^{\mf n^+}} = 0$, and since $V_\lambda^{\mf n^+}$ is a non-trivial subspace of $V_\lambda$, irreducibility of $V_\lambda$ implies that $\ker T = V_\lambda$, so $T=0$.\\
Now assume that $\sO(X,\sE)\neq\{0\}$. To prove surjectivity of $\varphi_\lambda$ and the additional result on highest weights occurring in $\Poly(\mf n^-,E)$, we first prove that $\sP(\mf n^-,E)$ is contained in the image of $\iota_\sN$. Let $f\in\sO(X,\sE)$ be the highest weight vector. Then, $p_f=\iota_\sN(f)$ is an element of $\sP(\mf n^-,E)$, and since $f(z) = p_f(q(z))$ must be holomorphic, it follows that $p_f$ is constant. Since $\sE$ is assumed to be irreducible, the $K$-equivariance of $\iota_\sN$ therefore implies that $\Im\iota_\sN$ contains all constant polynomials of $\sP(\mf n^-,E)$. An arbitrary polynomial $p\in\sP(\mf n^-,E)$ can be written as $p = \sum p_i e_i$ for $e_i\in E$ and $p_i\in\Poly(\mf n^-)$. Due to Corollary~\ref{cor:nearlyholomorphicgenerators}, there exist nearly holomorphic functions $g_i\in\sN(X)$ such that $g_i(z) = p_i(q(z))$. Now, set $f:=\sum g_if_i$, where $f_i$ denotes the holomorphic section corresponding to the constant $e_i$. Then, $f$ is nearly holomorphic (Corollary~\ref{cor:NearHolModule}), and $p = \iota_\sN(f)$. Therefore, $\sP(\mf n^-,E)$ is indeed contained in the image of $\iota_\sN$. Firstly, this implies that $\iota_\sN$ induces a bijection between $U$-highest weight vectors $f_\lambda$ of weight $\lambda$ in $\sN(X,\sE)$ and $K$-highest weight vectors $p_\lambda$ in $\Poly(\mf n^-,E)$ of the same weight. In particular, all highest weights occurring in $\Poly(\mf n^-,E)$ must be dominant integral for $\Phi^+$. Secondly, recall that any $S\in\Hom_K(V_\lambda^{\mf n^+},\Poly(\mf n^-,E))$ is uniquely determined by the image $p_\lambda = Sv_\lambda$ of the highest weight vector $v_\lambda$ of $V_\lambda^{\mf n^+}$. Since $v_\lambda$ is also the highest weight vector of $V_\lambda$, there exists a homomorphism $T\in\Hom_U(V_\lambda,L^2(X,\sE))$ such that $Tv_\lambda = f_\lambda$, where $f_\lambda$ is the $U$-highest weight vector determined by $p_\lambda = \iota_\lambda f_\lambda$. We therefore conclude that $S = \varphi_\lambda(T)$, which shows that $\varphi_\lambda$ is surjective.
\end{proof}

\begin{remark}
	A more explicit formula for the isomorphism $\varphi_\lambda$ is given in
	Section~\ref{sec:taylor}. In the following, irreducible subrepresentations of a representation 
	of $U$ (resp. $K$) are called $U$-\emph{types} (resp. $K$-\emph{types}). Then, 
	Theorem~\ref{thm:highestweights} states that there is a bijection between $U$-types in 
	$L^2(X,\sE)$ and $K$-types in $\Poly(\mf n^-,E)$. As its proof shows, this bijection is given by 
	an explicit correspondence between highest weight vectors. 
\end{remark}

\begin{corollary}\label{cor:highestweightcorrespondence}
	Assume that $\sO(X,\sE)\neq\{0\}$, and fix $\lambda\in\Lambda$. Then, $f\in L^2(X,\sE)$ is a 
	$U$-highest weight vector of weight $\lambda$ if and only if $f(z) = p(q(z))$ for all
	$z\in\mf n^+$ and $p\in\Poly(\mf n^-,E)$ is a $K$-highest weight vector of weight $\lambda$.
\end{corollary}

\begin{remark}
	In the case of the trivial line bundle, $\sE = X\times\CC$, Theorem~\ref{thm:highestweights}
	implies that $U$-types in $L^2(X)$ correspond bijectively to $K$-types in $\sP(\mf n^-)$. In
	particular, $L^2(X)$ decomposes multiplicity free if and only if $\sP(\mf n^-)$ does so. On the
	one hand, the decomposition of $L^2(X)$ into $U$-types is known from the Cartan--Helgason
	theorem, see \cite[V\S4]{He84}. On the other hand, the $L$-type decomposition of $\sP(\mf n^-)$
	is known due to the work of Hua (classical, \cite{Hu63}), Kostant (unpublished), and Schmid
	\cite{Sc69}. We thus obtain that the Cartan--Helgason theorem (applied to irreducible Hermitian 
	symmetric spaces of compact type) is equivalent to the Hua--Kostant--Schmid decomposition of the 
	polynomial algebra $\Poly(\Vm)$.
\end{remark}

For general vector bundles, the Hua--Kostant--Schmid decomposition can be used to obtain more explicit results of the weights and multiplicities occurring in the $U$-type decomposition of $L^2(X,\sE)$. We therefore briefly recall the details of the Hua--Kostant--Schmid decomposition.

We choose the lexicographic order of roots corresponding to the simple roots $(\alpha_1,\ldots, \alpha_\ell)$, i.e., $\alpha>0$ if $\alpha = \sum m_i\alpha_i$ with $m_i>0$ for the first non-zero coefficient. In particular, we have $\alpha>\beta$ for all $\alpha\in\Phi_{nc}^+$, $\beta\in\Phi_c$. Let $(\gamma_1,\ldots,\gamma_r)$, be the maximal system of strongly orthogonal roots, such that $\gamma_i$ is the \emph{highest} element of $\Phi_{nc}^+$ strongly orthogonal to $\gamma_j$ for $j<i$, i.e., $\gamma_i\pm\gamma_j$ are no roots.\footnote{We note that our system of strongly orthogonal roots differs from the system $(\tilde\gamma_1,\ldots,\tilde\gamma_r)$ originally defined by Harish-Chandra \cite{HC56} (which is used in \cite{Sc69}) by $\gamma_i = w_0\tilde\gamma_i$, where $w_0$ is the longest element of the Weyl group of $\Phi_c$.}

\begin{theorem}[Hua--Kostant--Schmid \cite{FK90,Hu63,Sc69}]\label{thm:HSK}
	The polynomial algebra $\Poly(\mf n^-)$ decomposes under the action of 
	$L$ multiplicity free into
	\[
		\Poly(\mf n^-) = \bigoplus_{\b m\in\NN^r_\geq}\Poly_\b m(\mf n^-)\,,
	\]
	where $\NN^r_\geq:=\set{\b m =(m_1,\ldots,m_r)\in\NN^r}{m_1\geq\cdots\geq m_r\geq 0}$, and
	$\Poly_\b m(\mf n^-)$ is an irreducible $L$-module of highest weight
	\[
		\gamma_\b m := m_1\gamma_1+\cdots+m_r\gamma_r
	\]
	with respect to $\Phi_c^+$.
\end{theorem}

\begin{remark}\label{rmk:ExplicitHighestWeights}
	There are well-known Jordan theoretic formulas for the highest weight vectors in
	$\Poly(\mf n^-)$ due to Upmeier \cite{Up86}, essentially involving the Jordan pair determinant 
	$\Delta$. According to Corollary~\ref{cor:highestweightcorrespondence}, this yields explicit 
	formulas for the highest weight vectors in $L^2(X)$. Explicitly, let 
	$(X_{\gamma_i},H_{\gamma_i},Y_{\gamma_i})$ denote the components of the $\mf{sl}_2$-triple 
	corresponding to $\gamma_i$, then
	\[
		p_\b m(w):= \Delta_1(w)^{m_1-m_2}\cdots\Delta_{r-1}(w)^{m_{r-1}-m_r}\cdot\Delta_r(w)^{m_r}
	\]
	is the highest weight vector of weight $\gamma_\b m$, where 
	$\Delta_i(w):=\Delta(e_i^+,e_i^- - w)$ with 
	$e_i^+ := X_{\gamma_1}+\cdots+X_{\gamma_i}$,
	$e_i^- := -Y_{\gamma_1}-\cdots-Y_{\gamma_i}$. Consequently,
	\[
		f_\b m(z) := p_\b m(q(z))
	\]
	with $z\in\mf n^+\subseteq X$ is the local picture of the highest weight vector of $L^2(X)$ with 
	weight $\gamma_\b m$.
\end{remark}

Now, for general vector bundles $\sE = G\times_P E$ the space $\Poly(\mf n^-,E)$ of $E$-valued complex polynomials on $\mf n^-$ is canonically isomorphic to the tensor product $\Poly(\mf n^-)\otimes E$, and this isomorphism is $L$-equivariant. Recall, that the highest weights occurring in a tensor product $E_\lambda\otimes E_\mu$ (with $\lambda,\mu\in\Lambda_c$) all have the form $\lambda+\nu$, where $\nu$ is a weight of $E_\mu$, and the multiplicity of $\lambda+\nu$ in $E_\lambda\otimes E_\mu$ is bounded by the dimension of the weight space $(E_\mu)^\nu\subseteq E_\mu$, see e.g.\ \cite[\S\,3]{Ku10}. Therefore, Theorem~\ref{thm:highestweights} and the Hua--Kostant--Schmid decomposition yield the following result. 

\begin{corollary}\label{cor:weightdescription}
	Assume that $\sO(X,\sE)\neq\{0\}$. Then, there is a bijection between $U$-types in $L^2(X,\sE)$ 
	and $K$-types in
	\[
		\bigoplus_{\b m\in\NN^r_\geq}\Poly_\b m(\Vm)\otimes E\,.
	\]
	In particular, all highest weights occurring in $L^2(X,\sE)$ are of the form $\lambda = 
	\gamma_\b m + \mu$, where $\mu$ is a weight in $E$, and its multiplicity, $m^\sE_\lambda$, is 
	bounded by the dimension of $E$.
\end{corollary}

\begin{remark}\label{rmk:linebundles}
	If $\sE=\linebundle$ is a line bundle, i.e., $E$ is 1-dimensional, and $\sE$ admits non-tivial 
	holomorphic sections, then Corollary~\ref{cor:weightdescription} yields that 
	$L^2(X,\linebundle)$ decomposes multiplicity free into
	\[
		L^2(X,\linebundle) = \widehat\bigoplus_{\b m\in\NN^r_\geq} V_{\gamma_\b m + \nu},
	\]
	where $\nu$ is the highest weight of $\rho:P\to E$, which is determined by the action of the 
	center of $L$ on $E$. This is a special case of Schlichtkrull's generalization of the 
	Cartan--Helgason theorem, see \cite{Sch84}. In our subsequent paper \cite{S12}, we obtain the 
	full version of Schlichtkrull's result by a more detailed analysis of nearly holomorphic 
	sections.
\end{remark}

\subsection{Application}
In this section, we apply our main result from the last section to the holomorphic tangent bundle $T^{(1,0)}$. This is the $G$-homogeneous vector bundle corresponding to the adjoint action of $L$ on $\mf n^+$, so $E=\mf n^+$ and $\rho(h) = \Ad_h$ for $h\in L$. As above, we set $hv := \Ad_h v$ for $h\in L$ and $v\in\mf n^+$. It is well-known that $T^{(1,0)}$ admits non-trivial holomorphic sections, e.g.\ this follows from the Borel-Weil Theorem and the observation that the highest weight of $\mf n^+$ is the highest root of $\mf g$, which is dominant integral for $\Phi^+$, see \cite[VI.1.8]{Bou02}. Therefore, in order to determine the multiplicities $m_\lambda$ of the decomposition
\[
	L^2(X,T^{(1,0)})= \widehat{\bigoplus_{\lambda\in\Lambda}}\; m_\lambda V_\lambda,
\]
Corollary~\ref{cor:weightdescription} yields that is suffices to determine the $K$-type decomposition of $\Poly_\b m(\mf n^-)\otimes\mf n^+$. 

We first recall the following result on tensor products due to Kostant \cite[Lemma~4.1]{Kos59}, see also \cite{Ku10}. For an irreducible representation $E_\lambda$ of $K$ of highest weight $\lambda\in \Lambda_c$, let $\Phi(E_\lambda)$ be the set of all weights $\nu\in\mf h^*$ with non-trivial corresponding weight space $(E_\lambda)^\nu\subseteq E_\lambda$. Recall, that $\Delta_c=\{\alpha_2,\ldots,\alpha_\ell\}$ is the set of simple roots in $\Phi_c$, and the $\mf{sl}_2$-triple corresponding to $\alpha_i\in\Delta_c$ is denoted by $(X_i,H_i,Y_i)$.

\begin{proposition}\label{prop:tensormultiplicity}
	For any $\lambda,\mu,\nu\in\Lambda_c$, the multiplicity of $E_\nu$ in $E_\lambda\otimes E_\mu$ is
	given by
	\[
		\dim\Set{v\in(E_\mu)^{\nu-\lambda}}
						{X_i^{\lambda(H_i)+1}v = 0\text{ for all $i=2,\ldots,\ell$}}.
	\]
\end{proposition}
\begin{remark}
	In \cite{Kos59}, it is assumed that $K$ is semi-simple. Here, $K=Z(K)K_\ss$ is reductive with 
	1-dimensional center $Z(K)$ and semi-simple part $K_\ss$. Due to Schur's Lemma, the center acts 
	on $E_\lambda$ and $E_\mu$ by scalars $c_\lambda$ and $c_\mu$, and hence the action of $Z(K)$ on 
	$E_\lambda\otimes E_\mu$ is given by the scalar $c_\lambda\cdot c_\mu$. Therefore, 
	Proposition~\ref{prop:tensormultiplicity} is an immediate consequence of the semi-simple case.
\end{remark}

This result on the multiplicities in tensor products can be applied to the decomposition of $\Poly_\b m(\mf n^-)\otimes\mf n^+$. Recall that $\gamma_\b m = m_1\gamma_1+\cdots+m_r\gamma_r$ is the highest weight of $\Poly_\b m(\mf n^-)$ with $m_1\geq\cdots\geq m_r$. For convenience, set $m_0:=+\infty$, $m_{r+1}:=0$, and let $e_j$ denote the $j$'th standard basis vector of $\NN^r$.

\begin{samepage}
\begin{proposition}\label{prop:tensordecomp}
	For all $\b m\in\NN^r_\geq$, the $K$-type decomposition of the tensor product
	$\Poly_\b m(\Vm)\otimes\mf n^+$ is given by
	\[
		\Poly_\b m(\Vm)\otimes\mf n^+ = \bigoplus_{\lambda\in\Lambda_\b m(\mf n^+)} E_\lambda
	\]
	where
	\begin{align}\label{eq:TangentBundleHeighestWeights}
		\Lambda_\b m(\mf n^+)
		:= \Set{\gamma_\b m + \beta\in\Lambda_c}
			{\begin{aligned}
				&\beta\in\Phi(\mf n^+)\text{ such that }\beta+\alpha_i\notin\Phi(\mf n^+)\\ 
				&\text{for all $i=2,\ldots,\ell$ with $\gamma_\b m(H_i)=0$}\end{aligned}}.
	\end{align}
	In particular, $\gamma_{\b m+e_j}\in\Lambda_\b m(\mf n^+)$ if and only if $m_j<m_{j-1}$.
\end{proposition}
\end{samepage}
\begin{proof}
As already mentioned above, all highest weights occurring in the decomposition of $\Poly_\b m(\Vm)\otimes\mf n^+$ must be of the form $\gamma_\b m + \beta\in\Lambda_c$ for some $\beta\in\Phi(\mf n^+)$. Here we have to show that such a weight occurs as a highest weight of the tensor product, if and only if $\beta+\alpha_i\notin\Phi(\mf n^+)$ for all $i=2,\ldots,\ell$ with $\gamma_\b m(H_i)=0$, and that its multiplicity is one. We apply Proposition~\ref{prop:tensormultiplicity} to this setting, so $\lambda = \gamma_\b m$, $E_\mu =\mf n^+$, and $\nu = \gamma_\b m + \beta$. Recall that the weight spaces of $\mf n^+$ are precisely the root spaces of $\mf g$ corresponding to non-compact positive roots. Therefore, $\dim(\mf n^+)^\beta = 1$ for all $\beta\in\Phi(\mf n^+)$ and hence $\gamma_\b m+\beta$ occurs with multiplicity at most one. Now, Proposition~\ref{prop:tensormultiplicity} states that $\gamma_\b m+\beta\in\Phi_c$ occurs in the tensor product if and only if $X_i^{\gamma_\b m(H_i)+1}v=0$ for all $v\in(\mf n^+)^\beta$ and all $i=2,\ldots,\ell$. Since the action of $\mf k$ on $\mf n^+$ is the restriction of the adjoint action of $\mf k$ on $\mf g$, the vanishing of $X_i^{\gamma_\b m(H_i)+1}v$ is equivalent to the condition that $\beta + (\gamma_\b m(H_i)+1)\alpha_i$ is not a root of $\mf g$. For fixed $i$, consider the $\alpha_i$-chain through $\beta$ in the root system of $\mf g$, i.e., $\beta + k\alpha_i$ for $k=-a_i,\ldots,b_i$ with $a_i,b_i\in\NN$ maximal. Since $b_i=0$ if and only if $\beta + \alpha_i\notin\Phi(\mf n^+)$, it remains to show that
\[
	b_i<\gamma_\b m(H_i) + 1\iff b_i=0\text{ or }\gamma_\b m(H_i)\neq 0\,.
\]
The implication from left to right easily follows from the fact that $\gamma_\b m$ is dominant integral for $\Phi_c$, so $\gamma_\b m(H_i)\in\NN$. This also shows the converse implication for the case $b_i=0$. Now assume that $b_i>0$ and $\gamma_\b m(H_i)\neq 0$, i.e., $\gamma_\b m(H_i)>0$. For $b_i=1$, there is nothing to show. Since $\beta$ is a non-compact root and $\alpha_i$ is a compact root, \cite[Lemma~15]{HC55} implies that $a_i+b_i\leq 2$. Therefore, it remains to consider the case $b_i=2$ and $a_i=0$. Due to standard properties of root systems, we obtain $b_i=b_i-a_i = -\beta(H_i)$. Since $\gamma_\b m + \beta\in\Lambda_c$, it follows that $\gamma_\b m(H_i)+\beta(H_i)\in\NN$, so $b_i\leq\gamma_\b m(H_i)$, which yields $b_i<\gamma_\b m(H_i)+1$. This completes the proof of \eqref{eq:TangentBundleHeighestWeights}.\\
Finally, for $\beta = \gamma_j$, we note that $\gamma_\b m+\beta\in\Phi_c$ if and only if $\b m+e_j\in\NN^r_\geq$, i.e., if and only if $m_j<m_{j-1}$. Moreover, since $\gamma_j\pm\alpha_i$ cannot both be roots (see \cite[Lemma~12]{HC56}), we either have $b_i=0$ or $a_i=0$. In both cases, the arguments above show that $b_i<\gamma_\b m(H_i)+1$, hence $\gamma_\b m + \gamma_j\in\Lambda_\b m(\mf n^+)$. This yields the last statement.
\end{proof}

\begin{remark}
	We note that if the root system $\Phi$ is simply laced, then the second condition in 
	\eqref{eq:TangentBundleHeighestWeights} is always satisfied and it just remains the condition 
	that the highest weight is of the form $\gamma_\b m+\beta$ for some $\beta\in\Phi(\mf n^+)$, 
	which was mentioned above as a necessary condition for highest weights of the tensor product, 
	i.e.,
	\[
		\Lambda_\b m(\mf n^+)
			= \Set{\gamma_\b m + \beta\in\Lambda_c}{\beta\in\Phi(\mf n^+)}.
	\]
	Indeed, since at least two of the vectors $\beta$, $\beta\pm\alpha_i$ have different lengths, 
	$\beta\pm\alpha_i$ cannot both be roots of $\mf g$. Therefore, if
	$\beta+\alpha_i\in\Phi(\mf n^+)$ then $\beta-\alpha_i\notin\Phi(\mf n^+)$ and hence $a_i=0$ 
	(using the notation of the proof above). Since $b_i=b_i-a_i=-\beta(H_i)$ it now follows that 
	$b_i<\gamma_\b m(H_i) + 1$, hence $\gamma_\b m(H_i)>0$. The classification table in the appendix 
	shows which of the simple Hermitian symmetric spaces have simply laced root systems.
\end{remark}

\begin{theorem}\label{thm:tangentdecomp}
	The $U$-type decomposition of $L^2(X,T^{(1,0)})$ is given by
	\[
		L^2(X,T^{(1,0)})
			= \widehat\bigoplus_{\lambda\in\Lambda(\mf n^+)}\; m_\lambda\cdot V_\lambda
	\]
	with $\Lambda(\mf n^+) := \bigcup_{\b m\in\NN^r_\geq}\Lambda_\b m(\mf n^+)$, and
	for $\lambda\in\Lambda(\mf n^+)$,
	\begin{align*}
		m_\lambda = \begin{cases}
			\#\set{i\in\{1,\ldots,r\}}{m_i>m_{i+1}} 
				& ,\ \lambda = \gamma_\b m\text{ with }\b m\in\NN^r_\geq,\\
			1 & \text{, else.}
		\end{cases}
	\end{align*}
\end{theorem}
\begin{proof}
It follows from Corollary~\ref{cor:weightdescription} and Proposition~\ref{prop:tensordecomp} that $\Lambda(\mf n^+)$ is the set of highest weights having positive multiplicity in the decomposition of $L^2(X,T^{(1,0)})$. It therefore remains to prove the explicit formula for $m_\lambda$, $\lambda\in\Lambda(\mf n^+)$. Since each $\Poly_\b m(\Vm)\otimes\Vp$ decomposes multiplicity free, higher multiplicities can only occur in combination with differing indices $\b m,\b m'\in\NN^r_\geq$. Assume $\gamma_\b m + \beta = \gamma_{\b m'} + \beta'$ for some weights $\beta,\beta'\in\Phi(\mf n^+)$. Consider the subspace $\mf h':=\sum_{i=1}^r \CC H_{\gamma_i}$ of the Cartan subalgebra $\mf h$. It is well-known that the restriction of a non-compact positive root to $\mf h'$ is of the form $\gamma_i$, $\tfrac{1}{2}(\gamma_i+\gamma_j)$, or $\tfrac{1}{2}\gamma_i$ with $1\leq i<j\leq r$ (each restricted to $\mf h'$), and the only non-compact positive root with restriction $\gamma_i$ is $\gamma_i$ itself (see \cite{Mo64}). Applying this result to the identity $\gamma_\b m + \beta = \gamma_{\b m'} +\beta'$ with the assumption $\b m\neq\b m'$ yields that $\beta = \gamma_i$ and $\beta' =\gamma_j$ for some $i,j$ such that $\b m + e_i =\b m'+e_j$. Therefore, the last statement of Proposition~\ref{prop:tensordecomp} implies that the multiplicity of $\lambda = \gamma_{\widetilde{\b m}}$ is the number of indices $i\in\{1,\ldots, r\}$ such that $\widetilde{\b m}-e_i$ is an element of $\NN^r_\geq$.
\end{proof}


\section{Generalized Taylor expansion formula}\label{sec:taylor}
In this section, we solve the problem posed in Remark~\ref{rmk:coeffQuestion} for the case of  Hermitian symmetric spaces. More precisely, we consider local nearly holomorphic sections on $\Vp\subseteq X$ and determine their holomorphic coefficients in a non-recursive way.

On $\Vp$, let $\Psi$ be the Kähler potential defined in Lemma~\ref{lem:HermMetricAndKaehlerPot} with corresponding $q$-map
\[
	q: \Vp\to\Vm,\ z\mapsto \bar z^{-z}.
\]
Since the vector bundle $\sE$ is $G$-homogeneous, we may identify smooth local sections $f\in C^\infty(\Vp,\sE_\Vp)$ via \eqref{eq:LocalVsInducedPicture} with smooth maps $f:\Vp\to E$. Due to Proposition~\ref{prop:LocalDescription} a local nearly holomorphic section $f\in\sN(\Vp,\sE_\Vp)$ of degree $m:=\deg f$ has a unique expansion into
\begin{align}\label{eq:nhfdecomp}
	f(z) = \sum_{|\b i|\leq m} f_\b i(z)\cdot q(z)^\b i
\end{align}
with holomorphic coefficients $f_\b i\in\sO(\Vp,E)$. The goal is to determine a non-recursive formula for the coefficients $f_\b i$. For this, the crucial observation is the compatibility of the chosen $q$-map with certain differential operators. For the following, we fix some basis $(c_1,\ldots, c_n)$ of $\mf n^+$, let $(\tilde c_1,\ldots,\tilde c_n)$ denote the corresponding dual basis of $(\mf n^+)^+$, and identify $(\mf n^+)^*$ with $\mf n^-$ via the isomorphism \eqref{eq:KillingIsomorphism}.

On the one hand, consider the invariant Cauchy--Riemann operator $\iCR$, which maps local smooth functions $f\in C^\infty(\Vp,E)$ to functions $\iCR f\in C^\infty(\Vp,E\otimes \Vp)$. Let $\iCR_\ell:C^\infty(\Vp,E)\to C^\infty(\Vp,E)$ ($1\leq\ell\leq n$) be defined by
\[
	\iCR_\ell f=\left(\Id_E\otimes \tilde c_\ell\right)\iCR f\,,
\]
which are the operators already used in the proof of Proposition~\ref{prop:LocalDescription}.

On the other hand, we define $\delta_\ell:C^\infty(\Vp,E)\to C^\infty(\Vp,E)$ ($1\leq\ell\leq n$) by
\[
	\delta_\ell:=\del_{c_\ell} + \delbar_{Q_{\bar z}c_\ell}.
\]
On compactly supported local sections, this operator coincides (up to sign) with the action of $(c_\ell,0,0)\in\mf u_\CC=\mf g$ in the (complexified) representation $d\pi_\CC$, see Lemma~\ref{lem:localLiealgebraAction}. 

For the multi-index $\b i = (i_1,\ldots,i_n)\in\NN^n$, we also define the operators
\begin{align}\label{eq:multioperator}
	\iCR^\b i := \prod_{\ell=1}^n \iCR_\ell^{i_\ell}\,,\quad
	\delta^\b i:=\prod_{\ell=1}^n \delta_\ell^{i_\ell}\,.
\end{align}
Before giving the formula for the coefficients of nearly holomorphic sections by means of these operators, we note that they mutually commute (in particular, we may arrange the single operators in \eqref{eq:multioperator} in arbitrary order).

\begin{lemma}\label{lem:commutators}
	For $\b i,\b j\in\NN^r$,
	\[
		[\iCR^\b i,\iCR^\b j] = 0,\quad
		[\delta^\b i,\delta^\b j] = 0,\quad
		[\iCR^\b i,\delta^\b j] = 0.
	\]
\end{lemma}
\begin{proof}
It suffices to prove these relations for single indices, so we prove 
\[
	[\iCR_k,\iCR_\ell]=[\delta_k,\delta_\ell] = [\iCR_k,\delta_\ell] = 0
\]
for all $1\leq k,\ell,\leq n$. The first commutator vanishes according to \cite[Lemma~2.0]{S86}, cf.\ also Proposition~\ref{prop:BasisPropsNearHol}. To evaluate the commutators involving $\delta_\ell$, we note that it suffices to consider compactly supported functions $f\in C^\infty_c(\mf n^+,E)$. Therefore, $\delta_\ell = -d\pi_\CC(c_\ell,0,0)$ as already noted above, and since $\mf n^+\subseteq\mf u_\CC$ is abelian, it follows that $[\delta_k,\delta_\ell]=0$. The vanishing of the last commutator is a consequence of the $U$-invariance of the Cauchy--Riemann operator $\iCR$, 
\begin{align}\label{eq:Uinvariance}
	\iCR(\pi(u) f) = (\pi(u)\otimes du^*)\iCR f
\end{align}
for all $u\in U$, where $du^*$ denotes the induced action on vector fields on $X$, see Proposition~\ref{prop:BasisPropsNearHol}. Setting $u := \exp(tY)$ for $Y\in\mf u$, $t\in\RR$, and taking the derivative at $t=0$, the left hand side of \eqref{eq:Uinvariance} becomes $\iCR(d\pi(Y)f)$. To evaluate the right hand side, we note that locally $\iCR f(z) = \sum_{j=1}^n\iCR_jf(z)\otimes c_j$, which yields
\begin{align}\label{eq:invariancecondition}
	\iCR(d\pi(Y)f)
		= \sum_{j=1}^n \big(d\pi(Y)\iCR_jf\big)\otimes c_j + \iCR_jf\otimes[\tilde Y,c_j]\,,
\end{align}
where $\tilde Y$ is the holomorphic vector field corresponding to $Y\in\mf u$ and $[\tilde Y,c_j]$ denotes the commutator of $\tilde Y$ with the constant vector field $c_j$. Since \eqref{eq:invariancecondition} remains valid for the complexified representation $d\pi_\CC$, we may set $Y=(c_\ell,0,0)$ and obtain
\[
	\iCR(\delta_\ell f) = \sum_{j=1}^n (\delta_\ell\iCR_jf)\otimes c_j\,,
\]
since $[c_\ell,c_j] = 0$. Finally applying $(\Id_E\otimes\tilde c_k)$, this yields $\iCR_k(\delta_\ell f) = \delta_\ell(\iCR_k f)$, i.e., $[\iCR_k,\delta_\ell] = 0$.
\end{proof}

\begin{proposition}[Generalized Taylor expansion formula]\label{prop:TaylorSeries}
	For any $f\in\sN(\Vp,\sE_\Vp)$, the coefficient $f_\b i\in\sO(\mf n^+,E)$ of the expansion 
	\[
		f(z) = \sum_{|\b i|\leq \deg f} f_\b i(z)\,q^{\b i}(z)
	\]
	is given by
	\[
		f_\b i(z) = \sum_{\b j\in\NN^n} c_{\b i \b j}\,z^\b j\quad\text{with}\quad
		c_{\b i\b j} := \tfrac{1}{\b i!\b j!}\,\delta^\b j\iCR^\b i f(0)\,.
	\]
\end{proposition}
\begin{proof}
Since the coefficient $f_\b i$ is a holomorphic map on $\mf n^+$, it is clear that it admits a Taylor expansion of the form $f_\b i(z) = \sum_{\b j\in\NN^n} c_{\b i \b j}\,z^\b j$, where the coefficient $c_{\b i\b j}\in\CC$ is given by $c_{\b i\b j} = \frac{1}{\b j}\del^{\b j}f_\b i(0)$. Therefore, it suffices to note the following three facts, (i) the $q$-map vanishes at $0$, (ii) due to \eqref{eq:qvanishing} the operators $\delta_\b i$ satisfy
\[
	\delta^\b j\big(f_\b i\cdot q^\b i\big) = \big(\del^\b jf_\b i\big)\cdot q^\b i\,,
\]
and (iii) the definition of $\iCR$ and the $q$-map implies that
\begin{align*}
	\iCR^\b j\big(f_\b i\cdot q^\b i\big)
		= f_\b i\cdot\iCR^\b j q^\b i
		= f_\b i\cdot\begin{cases}
			\tfrac{\b i!}{(\b i-\b j)!}\,q^{\b i-\b j} & \text{if $\b j\leq\b i$,}\\
			0	& \text{otherwise,}
		\end{cases}
\end{align*}
see also the proof of Proposition~\ref{prop:LocalDescription}.
\end{proof}

As an application of Proposition~\ref{prop:TaylorSeries}, we finally obtain a more explicit description of the isomorphism $\varphi_\lambda$ defined in Theorem~\ref{thm:highestweights}.

\begin{corollary}\label{cor:explicitisom}
	Assume $\sO(X,\sE)\neq\{0\}$, and fix $\lambda\in\Lambda$. The isomorphism
	\begin{align*}
		\varphi_\lambda:\Hom_U(V_\lambda,L^2(X,\sE))\to\Hom_K(V_\lambda^{\mf n^+},\Poly(\Vm,E)),\
		T\mapsto \iota_{\sN}\circ T|_{V_\lambda^{\mf n^+}}
	\end{align*}
	is given by
	\[
		\varphi_\lambda(T): v\mapsto
			f^v(w):=\sum_{\b i\in\NN^n}\tfrac{1}{\b i!}\,\iCR^{\b i}(Tv)\at{\mf n^+}(0)\,w^\b i\,.
	\]
\end{corollary}

\begin{remark}
	Recall that if the space $\sN(X,\sE)$ of (global) nearly holomorphic sections is non-trivial, 
	then Theorem~\ref{thm:NearlyHolDens} states that $\sN(X,\sE)$ is dense in $C(X,\sE)$ with 
	respect to uniform convergence. It would be interesting to study sequences and series of nearly 
	holomorphic sections by means of the generalized Taylor expansion formula.
\end{remark}


\appendix

\section*{Appendix}

The classification of irreducible Hermitian symmetric spaces of non-compact type is well-known, see \cite{He01} (Lie theoretic) or \cite{Lo77} (Jordan theoretic). The following table collects some of the data of this classification. Here, $X=U/K$ is the irreducible Hermitian symmetric space,
\[
	\mf g = \mf u_\CC = \mf n^+\oplus\mf k_\CC\oplus\mf n^-\,,
	\quad
	n = \dim X = \dim\mf n^+\,,\quad r = \rank X\,,
\]
and $p$ is the structure constant defined by \eqref{eq:BergmanAndDeterminant}.\\[5mm]
{\small
\renewcommand{\arraystretch}{1.5}
\hspace*{-1.8cm}
\begin{tabular}{|m{1.5cm}|c|c|m{4.7cm}|c|c|c|c|}
	\hline
	Type & $U$ & $K$
	& Dynkin diagram of $\mf g$\newline with marked parabolic 
	& $\mf n^+$ & $n$ & $r$ & $p$\\\hline\hline
	$I_{r,s}$, $r\leq s$ \newline{\small(A III)} & $\SU_{r+s}$ & $S(\U_r\times U_s)$ 
		& \DynkinA
		& $\CC^{r\times s}$ & $rs$ & $r$ & $r+s$\tabularnewline\hline
	$II_k$ \newline{\small(D III)} & $\SO_{2k}^\RR$ & $U_k$
		& \DynkinD
		& $\CC^{k\times k}_\text{asym}$ & $\tfrac{k(k-1)}{2}$ & $[\tfrac{1}{2}k]$ & $2k-2$ \\[4mm]\hline
	$III_k$ \newline{\small(C I)} & $\Sp_k^\RR$ & $U_k$
		& \DynkinC
		& $\CC^{k\times k}_\text{sym}$ & $\tfrac{k(k+1)}{2}$ & $k$ & $k+1$\\\hline
	$IV_k$ \newline{\small(BD I)} & $\SO_{k+2}^\RR$ & $\SO_k^\RR\times\SO_2^\RR$
		& $k$ even:\newline\DynkinB\vspace*{2mm}\newline $k$ odd:\vspace{-1mm}\newline\DynkinBD
		& $\CC^k$ & $k$ & $2$ & $k$ \\[1cm]\hline
	$V$ \newline{\small(E III)} & $\mf u = \mf e_{6(-78)}$ & $\mf k = \mf{so}^\RR_{10} + \RR$
		& \DynkinEVI
		& $\mathbb O_\CC^{1,2}$ & $16$ & $2$ & $12$\\[4mm]\hline
	$VI$ \newline{\small(E VII)} & $\mf u = \mf e_{7(-133)}$ & $\mf k = \mf e_6+\RR$ 
		& \DynkinEVII
		& $\mathbb O_{\CC,\text{herm}}^{3\times 3}$ & $27$ & $3$ & $18$ \\[4mm]\hline
\end{tabular}
}


\bibliographystyle{amsplain}
\bibliography{bibdb}

\end{document}